\documentclass[12pt]{article}

\usepackage[caption=false]{subfig}
\usepackage{algpseudocode}
\usepackage{pgfplots}
\pgfplotsset{compat=1.18}
\usepackage{tikz}

\usepackage{multirow}
\usepackage[algo2e,ruled,linesnumbered,vlined]{algorithm2e}

\SetCommentSty{mycommfont}

\usepackage{textcmds}
\usepackage{comment}
\usepackage{mathtools}
\usepackage{tabularx}
\usepackage{amsthm}

\usepackage{amssymb}
\usepackage{amsmath}
\usepackage{color}
\usepackage{comment}
\usepackage{geometry}        
\usepackage{graphicx}
\usepackage[caption=false]{subfig}
\usepackage{hyperref}
\usepackage{cleveref}
\usepackage{minifp}
\usepackage{paracol}
\usepackage{mathtools}
\usepackage{array}
\usepackage{subfiles}

\let\originalleft\left
\let\originalright\right
\renewcommand{\left}{\mathopen{}\mathclose\bgroup\originalleft}
\renewcommand{\right}{\aftergroup\egroup\originalright}

\usepackage{accents}
\usepackage{url} 
\newtheorem{theorem}{Theorem}
\newtheorem{lemma}{Lemma}

\newtheorem{definition}{Definition}
\newtheorem{example}{Example}[section]

\newcommand\funding[1]{\protect\\ \hspace*{15.37pt}{\bfseries Funding:} #1}
\newcommand{\email}[1]{\protect\href{mailto:#1}{#1}}

\usepackage{lipsum}
\usepackage{amsfonts}
\ifpdf
  \DeclareGraphicsExtensions{.eps,.pdf,.png,.jpg}
\else
  \DeclareGraphicsExtensions{.eps}
\fi

\title{Bifurcation analysis of multiple limit cycles created in boundary equilibrium bifurcations in hybrid systems\thanks{Manuscript draft
\funding{Hong Tang was funded by the China Scholarship Council; David Simpson was supported by Marsden Fund contract MAU2209 managed by Royal Society Te Ap\={a}rangi.}}}

\author{
Hong~Tang \thanks{Department of Engineering Mathematics, University of Bristol, Bristol BS8 1TR, UK
  (\email{hong.tang@bristol.ac.uk}).}
\and 
Alan~Champneys \thanks{Department of Engineering Mathematics, University of Bristol, Bristol BS8 1TR, UK   (\email{A.R.Champneys@bristol.ac.uk}).}
\and
David J.W.~Simpson
\thanks{School of Mathematical and Computational Sciences,
Massey University,
Palmerston North, 4410,
New Zealand (\email{D.J.W.Simpson@massey.ac.nz}).}
}

\usepackage{amsopn}

\makeatletter
\newcommand*{\addFileDependency}[1]{
  \typeout{(#1)}
  \@addtofilelist{#1}
  \IfFileExists{#1}{}{\typeout{No file #1.}}
}
\makeatother

\usepackage[style =numeric,backend=bibtex]{biblatex}
\begin{document}
    
    \newcommand\cO{\mathcal{O}}
    \newcommand{\rD}{{\rm D}}
    \newcommand{\ee}{\varepsilon}
    \newcommand{\transpose}{{\sf T}}
    \newenvironment{Example}{%
	\begin{example}%
	}{%
	\end{example}%
	\hfill $\blacksquare$ \\
}
\newenvironment{keywords}{ \\
	\noindent\textbf{Keywords:}\quad
}{\par}
    \maketitle

    \begin{abstract}
        A boundary equilibrium bifurcation (BEB) in a hybrid dynamical system occurs when a regular equilibrium collides with a switching surface in phase space. This causes a transition to a pseudo-equilibrium embedded within the switching surface, but limit cycles (LCs) and other invariant sets can also be created and the nature of these is not well understood for systems with more than two dimensions. This work treats two codimension-two scenarios in hybrid systems of any number of dimensions, where the number of small-amplitude limit cycles bifurcating from a BEB changes. The first scenario involves a limit cycle (LC) with a Floquet multiplier $1$ and for nearby parameter values the BEB creates a pair of limit cycles. The second scenario involves a limit cycle with a Floquet multiplier $-1$ and for nearby parameter values the BEB creates a period-doubled solution. Both scenarios are unfolded in a general setting, showing that typical two-parameter bifurcation diagrams have a curve of saddle-node or period-doubling bifurcations emanating transversally from a curve of BEBs at the codimension-two point. The results are illustrated with three-dimensional examples and an eight-dimensional airfoil model. Detailed computational results show excellent agreement to the unfolding theory and reveal further interesting dynamical features that remain to be explored.
 \\
    \end{abstract}

    \section{Introduction}
    \label{sec:intro}

    To understand the behaviour of a dynamical system, it is helpful to identify {\em bifurcations}
    where the dynamics changes in a fundamental way. There is a well established and extensive theory for bifurcations that has proved to be highly successful at explaining the dynamical behaviour of diverse systems in science, engineering, and medicine (and beyond), in part because much of the theory places no restriction on the system dimension, i.e.~the number of variables. This occurs when the pertinent dynamics belongs to a centre manifold that is one or two dimensional for most common bifurcations \cite{Yuri98}.

But this dimension reduction requires that the equations of motion are smooth. 
For piecewise-smooth ODEs (ordinary differential equations) and hybrid systems,
many basic bifurcations do not involve a centre manifold. The phase spaces of such systems have codimension-one switching surfaces where, in the case of piecewise-smooth systems, the functional form of the ODEs changes \cite{Bernardo2007,Fili88,Je18b}. Hybrid systems have reset laws (maps) that specify an instantaneous change to the state of system whenever it reaches a switching surface \cite{Goe12,VaSc00}. These systems are used to model a wide variety of mechanical and control systems, particularly vibro-impacting systems with hard unilateral constraints and systems that combine digital and analog elements \cite{Bro16,Iv00,WiDe00}.

The simplest type of bifurcation unique to piecewise-smooth and hybrid systems is a {\em boundary equilibrium bifurcation} (BEB) whereby a regular equilibrium (zero of some smooth component of the ODEs) collides with a switching surface. As shown by di Bernardo {\em et al.}~\cite{DiBernardo2008}, if we look only at equilibria, there are two cases for the local dynamics occurring as we cross the BEB. Either the regular equilibrium is replaced by another equilibrium ({\em persistence}), or it collides and annihilates with a coexisting equilibrium ({\em nonsmooth fold}). For hybrid and Filippov systems the second equilibrium is a {\em pseudo-equilibrium} on the switching surface.

But BEBs readily generate other dynamics. For two-dimensional systems this is well understood; we refer to \cite{FrPo98,Simpson2012} for continuous, non-differentiable ODEs, \cite{Gl16d,HoHo16,KuRi03} for Filippov systems, and \cite{DiBernardo2008,Si22b} for hybrid systems. However, the dynamical complexity increases with dimension \cite{GlJe15} and already in three dimensions BEBs can create chaotic attractors regardless of whether the system is continuous \cite{Simpson2016}, Filippov \cite{CaNo20,Gl18,NoTe19,Simpson2018b}, or hybrid \cite{DiBernardo2008}.

Recent work of two of the authors \cite{HoCh23} revealed a BEB with unusual characteristics in an eight-dimensional hybrid system modelling the motion of an airfoil. The system has a stable regular equilibrium corresponding to the airfoil in a fixed position. As the air velocity (a model parameter) is increased, the equilibrium undergoes a BEB whereby the airfoil flap hits a clamped stop. This bifurcation is of persistence type, leading to a stable pseudo-equilibrium where the flap is in contact with the stop. But the BEB also creates a stable limit cycle, and this is particularly interesting because the system then has multiple attractors and because such a transition is not possible for BEBs in two dimensions.

In this paper, to understand how such a BEB can arise we consider curves of BEBs in two-dimensional parameter space. Qualitative changes to the dynamics associated with the BEB occur at codimension-two points, such as where the regular equilibrium is non-hyperbolic. It is known that if the equilibrium has a zero eigenvalue then generically a curve of saddle-node bifurcation emanates from the curve of BEBs with a quadratic tangency, while if it has a pair of purely imaginary eigenvalues then generically a curve of Hopf bifurcations emanates transversally from the curve of BEBs \cite{DiPa08,GuSe11,Si10,SiKo09}. But in the airfoil model it is the limit cycle that is non-hyperbolic at a nearby codimension-two point. Such codimension-two BEBs do not appear to have been studied before, possibly being overlooked because they require at least three dimensions to occur generically.

Specially, we unfold two codimension-two scenarios where a BEB creates a non-hyperbolic limit cycle. We show that if the limit cycle has a Floquet multiplier of $1$ then in two-dimensional parameter space a curve of saddle-node bifurcations (of limit cycles) emanates transversely from a curve of BEBs. On one side of the codimension-two point, the BEB creates two limit cycles, one of which can be stable and coexisting with a stable pseudo-equilibrium as observed in the airfoil model. We stress that the saddle-node and BEB curves meet transversally, which is different to previously reported cases where the saddle-node bifurcations are of equilibria \cite{DeDe12,SiKo09}. We also show that if the limit cycle has a Floquet multiplier of $-1$, then a curve of period-doubling bifurcations emanates transversely from the curve of BEBs. These unfoldings apply to impacting hybrid systems of three or more dimensions.

The remainder of this paper is organized as follows. \Cref{sec:hybrid} introduces general notation for impacting systems and BEBs, then \Cref{sec:main_results} presents the main results proved using blow-up and centre-manifold techniques. \Cref{sec:examples} illustrates the results with three-dimensional examples and the airfoil model, then \Cref{sec:conc} draws conclusions and suggests avenues for future work.


\section{Impacting hybrid systems}
    \label{sec:hybrid}

Many vibro-impacting mechanical systems are well-modelled by
treating objects as rigid and impact events as instantaneous.
With this viewpoint, impacts bring about a discontinuous change to the velocity of objects.
This leads to mathematical models that are hybrid systems,
combining ODEs for the smooth evolution in open intervals of time  with reset laws for impact events.
Such {\em impacting hybrid systems} usually have the property that
as the impact velocity goes to zero, the recoil velocity also goes to zero.
It follows that for grazing events (i.e.~zero-velocity impacts)
nearby trajectories stay near each other regardless of whether or not an impact occurs.
Such continuity of the flow does not occur for
more general hybrid systems such as impulsive systems \cite{HaCh06}.

Following \cite{Bernardo2007,HoCh23} we consider an impacting hybrid system that in local
co-ordinates close to a single impacting surface can be written in the form 
    \begin{equation}
        \begin{split}
            \dot{x} &= F(x;\xi), \quad \text{for $H(x;\xi) > 0$}, \\
            x &\mapsto R(x;\xi), \quad \text{for $H(x;\xi) = 0$},
        \end{split}
        \label{eq:hybridSystem}
    \end{equation}
where the vector field $F$, the reset law $R$, and the switching function $H$ are $C^k$ functions (with $k$ sufficient large) of the state variable $x \in \mathbb{R}^n$ and a parameter vector $\xi \in \mathbb{R}^m$. When considering codimension-two bifurcations we shall often choose $m=2$ and    
set $\xi=(\mu,\eta)$, where the BEB occurs for $\mu=0$. Where it is unequivocal, we shall sometimes suppress the parameter dependence altogether. 

Let us describe the basic features of the dynamics of
an impacting hybrid system \eqref{eq:hybridSystem}.
\Cref{fig:3D_DS_traj_interaction} illustrates the path of a typical trajectory in such a system. 
     \begin{figure}[ht!]
        \centering
            \includegraphics[width = 0.6 \linewidth]{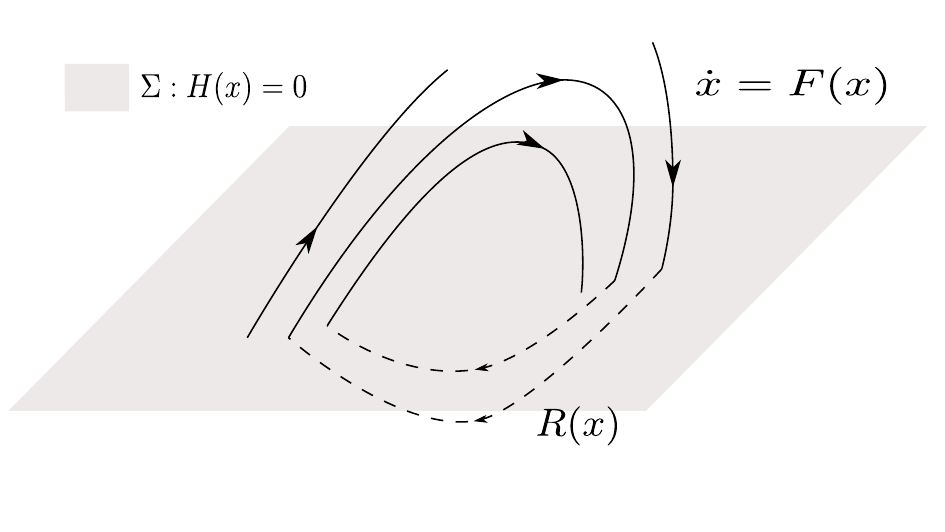}
        \caption{A three-dimensional sketch showing a trajectory of \eqref{eq:hybridSystem} intersecting the switching surface $\Sigma$ several times.}
        \label{fig:3D_DS_traj_interaction}
    \end{figure}
The trajectory follows the vector field $F$ until reaching the switching surface
$\Sigma = \left\{ x \in \mathbb{R}^n \,|\, H(x) = 0 \right\}$
where it is mapped under $R$ to a different point on $\Sigma$ then re-commences evolution under $F$.

We assume $\Sigma = \left\{ x \in \mathbb{R}^n \,|\, H(x) = 0 \right\}$ is a smooth codimension-one surface, so assume the gradient vector $\nabla H(x)$ is non-zero at all points on $\Sigma$.
The direction of $F(x)$ relative to $\Sigma$ is governed by the sign of the first Lie derivative
    \begin{equation}
    \label{eq:normal-incoming-velocity}
        v(x) = \mathcal{L}_F H(x) = \nabla H(x)^\transpose F(x),
    \end{equation}
    which is in keeping with mechanical systems we refer to as the velocity with respect to the impact surface.
    
We use this derivative to partition $\Sigma$ into three pieces,
the {\em incoming set} $\Sigma^- =
    \left\{ x \in \Sigma \,|\, v(x) < 0 \right\}$,
    the {\em outgoing set} $\Sigma^+ = \left\{ x \in \Sigma \,|\, v(x) > 0 \right\}$,
    and the {\em grazing set} $\Sigma^0 = \left\{ x \in \Sigma \,|\, v(x) = 0 \right\}$.

The reset law $R$ models impact events
where the recoil velocity tends to zero as the impact velocity tends to zero.
Hence we can write the reset law as
    \begin{equation}
    R(x) = x + W(x) v(x),
    \label{eq:R}
    \end{equation}
where $W$ is a smooth vector field, so that $R$ is the identity map on the grazing set. Note that by writing $v(R(x)) = -r(x)v(x)$, the
quantity $r(x)$ can be interpreted as a {\em coefficient of restitution}.
We assume $r(x) > 0$ for all $x \in \Sigma^- \cup \Sigma^0$
so that the reset law maps $\Sigma^-$ to $\Sigma^+$.
It is a simple exercise to show $r(x) = -1 - \nabla v(x)^{\transpose} W(x)$ at any $x \in \Sigma^0$.
Consequently
\begin{equation}
    \nabla v(x)^{\transpose} W(x) < -1
    \label{eq:NvW}
\end{equation}
at any $x \in \Sigma^0$ which will be important below.

It is useful to also define the acceleration relative to $\Sigma$: 
    $$
    a(x) = \mathcal{L}^2_F H(x) = \nabla v(x)^\transpose F(x).
    $$
If a trajectory hits $\Sigma^0$ at a point where $a(x)>0$ then
it {\em grazes} $\Sigma$ with a quadratic tangency. There is a rich literature on grazing bifurcations in impacting systems via the theory of discontinuity mappings
\cite{DiBu01,FrNo97,FrNo00,No91}.
 
Points that instead have $a(x)<0$ define the {\em sticking set}
$$
\Sigma^0_- = \left\{ x \in \Sigma^0 \,|\, a(x) < 0 \right\}.
$$ 
Typically trajectories can only approach $\Sigma^0_-$
via the accumulation of an infinite {\em chattering sequence} 
of impacts (sometimes referred to as a Zeno phenomenon)
\cite{JoEg99,ZhJo01}.
Once such accumulation has occurred, a consistent way within the piecewise-smooth framework to define future evolution
is to adapt the Filippov formalism \cite{Fili88}
and define the {\em sticking vector field}
    \begin{equation}
        F_s(x) = \left( I - \frac{W(x) \nabla v(x)^\transpose}{\nabla v(x)^\transpose W(x)} \right) F(x),
        \label{eq:Fs}
    \end{equation}
    which is tangent to $\Sigma^0$.
    {\em Sticking motion} refers to evolution on the codimension-two
    sticking set $\Sigma^0_-$ governed by $\dot{x} = F_s(x)$.
Notice \eqref{eq:Fs} is well-defined by \eqref{eq:NvW}.

This paper concerns bifurcations of equilibria,
which, for impacting hybrid systems of the form \eqref{eq:hybridSystem},
may be of two types: regular equilibria and pseudo-equilibria.
These are zeros of $F$ and $F_s$, respectively,
and are valid solutions to \eqref{eq:hybridSystem}
if they are `admissible', as defined below.

    \begin{definition}
    \label{def:1}
A {\em regular equilibrium} is a point $x^* \in \mathbb{R}^n$ for which $F(x^*) = 0$; it is {\em admissible} if $H(x^*) > 0$, {\em virtual} if $H(x^*) < 0$, and a {\em boundary equilibrium} if $H(x^*) = 0$. A {\em pseudo-equilibrium} is a point $x^* \in \Sigma^0$ for which $F_s(x^*) = 0$; it is {\em admissible} if $a(x^*) < 0$,
        and {\em virtual} if $a(x^*) > 0$.
    \end{definition}

\subsection{Boundary equilibrium bifurcations}
    \label{sec:beb}

As the parameter vector $\xi$ of an impacting hybrid system
\eqref{eq:hybridSystem} is varied,
a boundary equilibrium bifurcation (BEB)
occurs when a regular equilibrium $x^*(\xi)$ collides with
the switching surface $\Sigma$.
At the bifurcation, $x^*$ is a boundary equilibrium,
and also a pseudo-equilibrium with $a(x^*) = 0$
because all the Lie derivatives $\mathcal{L}^i_F H(x^*)$, for $i=1,2,\ldots$, are zero.

Consequently BEBs can be thought of as
a transition between regular and pseudo-equilibria,
and generically this transition conforms to one of two cases \cite{DiBernardo2008}.
In both cases the equilibria are unique, locally,
and each is admissible on one side of the bifurcation
and virtual on the other side of the bifurcation.
If they are admissible on different sides of the bifurcation
the BEB is referred as {\em persistence}
because if we only consider admissible features of the dynamics
then a single equilibrium is seen to persist.
If instead the equilibria are admissible on the same side
of the bifurcation the BEB is referred to as a {\em nonsmooth fold}
because, with the same mindset,
two equilibria are seen to collide and annihilate
much like a saddle-node bifurcation or fold.

While a complete characterisation of equilibria near generic BEBs
is available, except for two-dimensional systems \cite{DiBernardo2008} there is
very little theory on the existence and nature of other invariant sets
that can be created in BEBs.
Recent work of two of the authors \cite{HoCh23}
gave sufficient conditions for single-impact limit cycles to be born at a BEB, and a method to analyse their stability. In particular it was found that for
systems with more than two dimensions, multiple limit cycles can be born at a BEB
and a stable limit cycle can coexist with an admissible stable pseudo-equilibrium.
This motivates the present study on codimension-two bifurcations that change the number of limit cycles being born at a BEB.

Consider a system \eqref{eq:hybridSystem} with
parameter vector $\xi = (\mu,\eta) \in \mathbb{R}^2$
that undergoes a BEB when $\mu = 0$ for all values of $\eta$
(which later will be used to unfold codimension-two points).
For simplicity, suppose near-identity transformations have been made to flatten the switching surface $\Sigma$ so that  
the switching function $H$ can be written as
$$
    H(x;\mu,\eta) = C^\transpose x,
$$
for a constant $n\times 1$ vector $C$.
Since the vector fields $F$ and $W$ are smooth,
where $W$ appears in the reset law via \eqref{eq:R},
and $F(0;0,\eta) = 0$ for all $\eta$, we can write
    \begin{equation}
        \begin{split}
            F(x;\mu,\eta) &= A(\eta) x + M(\eta) \mu + \cO \left( \left( \|x\| + |\mu| \right)^2 \right), \\
            W(x;\mu,\eta) &= -B(\eta) + \cO \left( \|x\| + |\mu| \right),
        \end{split}
        \label{eq:IHS_series}
    \end{equation}
    where $A \in \mathbb{R}^{n \times n}$ and $B, M \in \mathbb{R}^n$
    are smooth functions of $\eta$.
    By substituting these into our earlier formulas, we obtain
    \begin{align}
        R(x;\mu,\eta) = (I - B C^\transpose A) x - B C^\transpose M \mu + \cO \left( \left( \|x\| + |\mu| \right)^2 \right).
    \end{align}

If $A$ is invertible, then, locally, \eqref{eq:hybridSystem} has a unique regular equilibrium at
    \begin{equation}
        x_{\rm reg}(\mu,\eta) = -A^{-1} M \mu + \cO \left( \mu^2 \right).
        \label{eq:regularEquilibrium}
    \end{equation}
    This equilibrium is admissible if $H(x_{\rm reg}(\mu,\eta);\mu,\eta) > 0$, where
    \begin{equation}
        H(x_{\rm reg}(\mu,\eta);\mu,\eta) = -C^\transpose A^{-1} M \mu + \cO \left( \mu^2 \right).
        \label{eq:xRegAdm}
    \end{equation}
    The eigenvalues associated with $x_{\rm reg}$ are $O(\mu)$ perturbations of the eigenvalues of $A$.
    So if all eigenvalues of $A$ have negative real part,
    then $x_{\rm reg}$ is admissible and asymptotically stable for all sufficiently small values of $\mu$
    with the sign of $\mu$ opposite to the sign of $C^\transpose A^{-1} M$.

By substituting the series \eqref{eq:IHS_series} into the sticking vector field \eqref{eq:Fs}, we obtain
    \begin{equation}
        F_s(x) = \left( I - \frac{B C^\transpose A}{C^\transpose A B} \right) (A x + M \mu) + \cO \left( \left( \|x\| +
        |\mu| \right)^2 \right),
        \label{eq:Fs2}
    \end{equation}
where $C^\transpose A B > 0$ (because $\nabla v(x)^\transpose W(x) < 0$ at points on $\Sigma^0$, see \eqref{eq:NvW}).
The Jacobian matrix of $F_s$ at $(x;\mu) = (0;0)$ is
    $$
    A_s = \left( I - \frac{B C^\transpose A}{C^\transpose A B} \right) A.
    $$
This matrix has a zero eigenvalue with algebraic multiplicity at least two
(because the codimension-two sticking set is invariant under $F_s$).
If the algebraic multiplicity is exactly two (as is generically the case),
then, locally, \eqref{eq:hybridSystem} has a unique pseudo-equilibrium
$x_{\rm ps}(\mu,\eta)$ with $x_{\rm ps}(0,\eta) = 0$ for all $\eta$.
By  \cref{def:1} the pseudo-equilibrium is admissible if $a(x_{\rm ps}(\mu,\eta);\mu,\eta) < 0$,
and after routine algebraic manipulation \cite{DiBernardo2008} we obtain
    \begin{equation}
        a(x_{\rm ps}(\mu,\eta);\mu,\eta) = \frac{C^\transpose A B C^\transpose A^{-1} M}{C^\transpose A^{-1} B} \mu + \cO
        \left( \mu^2 \right).
        \label{eq:xPsAdm}
    \end{equation}

If we restrict our attention to the sticking set,
then the eigenvalues associated with $x_{\rm ps}$ are the $n-2$ non-zero eigenvalues of $A_s$.
These are the eigenvalues of an $(n-2) \times (n-2)$ matrix $\hat{A}_s$
defined as the Jacobian matrix of the restriction of $F_s$ to the sticking set.
So if all eigenvalues of $\hat{A}_s$ have negative real part,
then $x_{\rm reg}$ is an asymptotically stable equilibrium of $F_s$ on $\Sigma_0$.

The signs of \eqref{eq:xRegAdm} and \eqref{eq:xPsAdm}
determine the admissibility of the equilibria,
so we can use the coefficients of their leading order terms
to the classify the BEB as persistence or a nonsmooth fold.

\begin{theorem}[\cite{DiBernardo2008}]
\label{th:CAB}
Consider an impacting hybrid system \eqref{eq:hybridSystem}
written with \eqref{eq:IHS_series}.
Suppose $A$ is non-singular, $C^\transpose A B > 0$, and $C^\transpose A^{-1} M \ne 0$.
If $C^\transpose A^{-1} B < 0$ the BEB at $\mu = 0$ corresponds to persistence,
while if $C^\transpose A^{-1} B > 0$ it is a nonsmooth fold.
\end{theorem}

\paragraph{Remarks}
\begin{enumerate}
\item 
The proof of Theorem \ref{th:CAB} is based on a co-ordinate {\em blow-up}
that leads to a scale-invariant normal form.  
Specifically, for small $\mu \neq 0$, under $y = \frac{x}{\mu}$
the system \eqref{eq:hybridSystem} becomes, to leading order, 
    \begin{equation}
        \begin{split}
            \dot{y} &= A y + M\sigma , \qquad \text{for } C^\transpose y > 0, \\
            y &\mapsto (I - B C^\transpose A) y - B C^\transpose M \sigma,            
            \qquad \text{for } C^\transpose y =0,
        \end{split}
        \label{eq:hybridSystemScaledTruncated}
    \end{equation}
    where $\sigma = \mbox{sign}(\mu)$.
Thus, to leading order, the dynamics near a BEB are scale invariant for each sign of $\mu$. In what follows, we describe scenarios that require us to go beyond this leading-order form; nevertheless, the blown-up co-ordinate $y$ will prove invaluable. 
  \item 
The basic classification provided by  \cref{th:CAB} is intimately connected to the stability of the equilibria.
In particular, we find \cite[Eq.~50]{DiBernardo2008}
    \begin{equation}
        \det \big( \hat{A}_s \big) = -\frac{C^\transpose A^{-1} B}{C^\transpose A B} \,\det(A).
        \label{eq:DiNo08eqn50}
    \end{equation}
So if both equilibria are asymptotically stable,
then all eigenvalues of $A$ and $\hat{A}_s$ have negative real part,
    thus $\det(A)$ and $\det \big( \hat{A}_s \big)$ have the same sign
    (they are both negative if $n$ is odd and both positive if $n$ is even).
    In this case $C^\transpose A^{-1} B < 0$ because $C^\transpose A B >0$,
    so the boundary equilibrium bifurcation corresponds to persistence.
    More generally \eqref{eq:DiNo08eqn50} allows us to classify the
    bifurcation from the number of positive eigenvalues of $A$ and $\hat{A}_s$,
    as done originally by Feigin \cite{Feigin78} for a related non-smooth bifurcation of maps.
\end{enumerate}

Next we consider limit cycles of \eqref{eq:hybridSystem}.

\subsection{Single-impact limit cycles}
    \label{sec:poincare}

Suppose an impacting hybrid system \eqref{eq:hybridSystem}
has a single-impact period-$T$ limit cycle.
This limit cycle corresponds to a point $x^* \in \Sigma^+$ satisfying
\begin{equation}
  x^* = \varphi(R(x^*),T), \quad \mbox{ with }
H(\varphi(R(x^*),t)) \neq 0, \mbox{ for all } 0<t<T,
        \label{eq:composedReturnMap}
\end{equation}
where $\varphi$ denotes the flow of the vector field $F$.
The stability of the limit cycle can be determined from the derivative of $\varphi(R(x),T)$, which can be evaluated by multiplying
a {\em saltation matrix} with the monodromy matrix associated with $F$, \cite{DiBernardo2008,HoCh23}.
    
In \cite{HoCh23} we showed how to convert the problem of the existence of such a limit cycle of the truncated form \eqref{eq:hybridSystemScaledTruncated} into an eigenproblem involving the matrix exponential $\exp({ AT})$. Here we
treat the general system \eqref{eq:hybridSystem} and relate
limit cycles to fixed points of a Poincar\'e map.
Care is needed to ensure the map is well-defined arbitrarily
close to the BEB at $\mu = 0$ where the limit cycle shrinks to a point,
and we use blown-up coordinates to achieve this.

Without loss of generality, let us consider limit cycles
on the $\mu > 0$ side of the BEB.
Again we use the blow-up
\begin{equation}
y = \frac{x}{\mu},
\label{eq:blowUp}
\end{equation}
that for $\mu > 0$
converts \eqref{eq:hybridSystem} to the form
    \begin{equation}
        \begin{split}
            \dot{y} &= A y + M + \cO(\mu), \qquad \text{for } C^\transpose y > 0, \\
            y &\mapsto (I - B C^\transpose A) y - B C^\transpose M + \cO(\mu),           
            \qquad \text{for } C^\transpose y =0.
        \end{split}
        \label{eq:hybridSystemScaled}
    \end{equation}
Unlike in \eqref{eq:hybridSystemScaledTruncated} we have not truncated,
so \eqref{eq:hybridSystemScaled} is conjugate to \eqref{eq:hybridSystem}
for $\mu > 0$.
But although \eqref{eq:hybridSystemScaled} was constructed
subject to the assumption $\mu > 0$, {\em functionally} it is well-defined
and can be extended smoothly
for all values of $\mu$ in a neighborhood of $0$.
This observation is crucial as it allows 
us to use the implicit function theorem in the proofs below.
Note, we use tildes for quantities ($F$, $H$, $\Sigma$ etc) defined in terms
of $y$ rather than $x$. 

Let $Q$ be an $(n-1) \times n$ matrix with the property that
    $$
    S = \begin{bmatrix}
            C^\transpose \\ Q
    \end{bmatrix}
    $$
is invertible. Then $u = \zeta(y) = Q y$ is a bijection from 
$\tilde{\Sigma}$ to $\mathbb{R}^{n-1}$.
Next we define 
$$
\Pi^- = \zeta(\tilde{\Sigma}^-), \quad 
\Pi^+ = \zeta(\tilde{\Sigma}^+), \quad 
\Pi^0 = \zeta(\tilde{\Sigma}^0),
$$
and a Poincar\'e map $P : \Pi^+ \to \Pi^+$
for the system \eqref{eq:hybridSystemScaled} by
    \begin{equation}
        P(u;\mu,\eta) = \zeta \big( \tilde{R}(y^-) \big),
        \label{eq:P}
    \end{equation}
where $y^-$ is the first point at which the forward orbit of $\zeta^{-1}(u) \in \tilde{\Sigma}^+$ under $\tilde{F}$ reaches $\tilde{\Sigma}^-$, assuming it ever does so.
    This is illustrated in \Cref{fig:3D_poincare_section_returnMap}.
     \begin{figure}[ht!]
        \centering
        \vspace{-1em}
            \includegraphics[angle=90,width = 0.6 \linewidth]{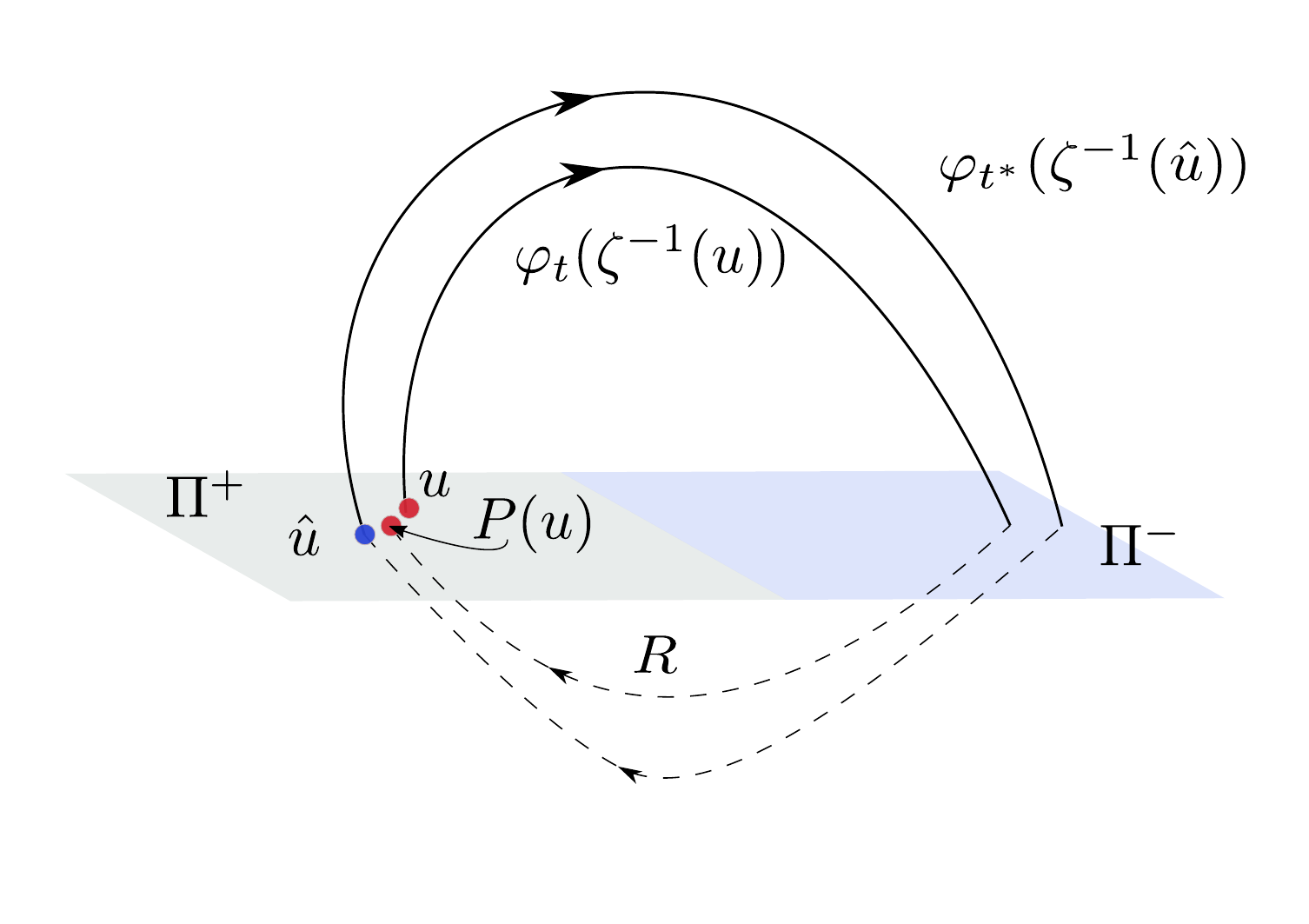}
        \vspace{-2em}
        \caption{A three-dimensional sketch illustrating the action of
        the Poincar\'e map $P$, defined by \eqref{eq:P}.
        A fixed point $\hat{u}$ of $P$ with $\mu > 0$ corresponds
        to a limit cycle of \eqref{eq:hybridSystem}.}
        \label{fig:3D_poincare_section_returnMap}
    \end{figure}
    If $\hat{u}$ is a fixed point of $P$ with $\mu > 0$,
    then $x^{*} = \mu \zeta^{-1}(\hat{u}) \in \Sigma^+$
    belongs to a single-impact periodic orbit of \eqref{eq:hybridSystem}.
    The eigenvalues of the Jacobian matrix $\rD P(u)$ are the non-trivial Floquet multipliers of the periodic orbit.
    The following technical result shows that $P$ is smooth.

    \begin{lemma}
        Suppose $F$ and $R$ in \eqref{eq:hybridSystem} are $C^k$ ($k \ge 2$).
        If $P$ is well-defined at $(u_0;\mu_0,\eta_0)$, where $u_0 \in \Pi^+$, $\mu_0 \in \mathbb{R}$, and $\eta_0 \in \mathbb{R}$,
        then it is well-defined and $C^{k-1}$ in a neighbourhood of $(u_0;\mu_0,\eta_0)$.
        \label{le:PisSmooth}
    \end{lemma}

    \begin{proof}
        The scaled system \eqref{eq:hybridSystemScaled} is $C^{k-1}$ because we have divided by $\mu$,
        so the vector field $\tilde{F}$ admits a $C^{k-1}$ flow $\varphi(y,t;\mu,\eta)$ \cite{Arnold73}.
        Thus $J(y,t;\mu,\eta) = \tilde{H}(\varphi(y,t;\mu,\eta))$ is a $C^{k-1}$ function.
        Let $y^+ = \zeta^{-1}(u_0)$ and $\tau$ be the evolution time between $y^+$ and $y^-$
        belonging to $\tilde{\Sigma}^-$.
        That is, $y^- = \varphi(y^+,\tau;\mu_0,\eta_0)$,
        and so $J(y^+,\tau;\mu_0,\eta_0) = 0$.
        Observe $\frac{\partial J}{\partial t}(y^+,\tau;\mu_0,\eta_0) = \nabla \tilde{H}(y^-)^\transpose \tilde{F}(y^-;\mu_0,\eta_0)$
        is negative because $y^-$ belongs to $\tilde{\Sigma}^-$.
        Thus, locally, we can solve $J(y,t;\mu,\eta) = 0$ for $t$:
        by the implicit function theorem there exists a unique $C^{k-1}$ function $T$ defined on a neighbourhood
        $\mathcal{U}$ of $(y^+;\mu_0,\eta_0)$
        such that $J(y,T(y;\mu,\eta);\mu,\eta) = 0$
        for all $(y;\mu,\eta) \in \mathcal{U}$.
        So the forward orbit of any such $y$ under $\tilde{F}$ returns to
        $\tilde{\Sigma}$ at the point $\varphi(y,T(y;\mu,\eta);\mu,\eta)$.
        Moreover, we can assume $\mathcal{U}$ has been taken small enough
        that this is the first point at which the orbit returns to $\tilde{\Sigma}$.
        So for any $(u;\mu,\eta)$ such that
        $\left( \zeta^{-1}(u);\mu,\eta \right) \in \mathcal{U}$,
        the Poincar\'e map $P(u;\mu,\eta)$ is well-defined and
        $P(u;\mu,\eta) = \zeta \left( \tilde{R} \left( \varphi(y,T(y;\mu,\eta);\mu,\eta) \right) \right)$, where $y = \zeta^{-1}(u)$.
        Thus $P$ is $C^{k-1}$ because we have expressed it as a composition of $C^{k-1}$ functions.
    \end{proof}

In summary we have reduced questions of existence
and stability of single-impact limit cycles to
those of fixed points of a Poincar\'e map $P$.
Importantly $P$ is smooth so we can employ classical (smooth)
bifurcation theory to analyse bifurcations of the limit cycles.

\section{Main results}
\label{sec:main_results}

We shall now assume that a fixed point of $P$, defined by \eqref{eq:P}
exists at parameter values $(\mu,\eta)=(0,0)$.
As discussed above, $P$ is well-defined and smooth
for all values of $\mu$ in a neighbourhood of $0$
even though it only corresponds to dynamics of \eqref{eq:hybridSystem} for $\mu > 0$,
where this choice of sign was taken without loss of generality.
We then allow $\eta$ to unfold a degeneracy in the linearization about this fixed point.

    \subsection{BEB-saddle-node bifurcation}

We first consider the case where  
    the Poincar\'e map $P$ has a fixed point with stability multiplier $1$
at $(\mu,\eta) = (0,0)$.
Such a multiplier suggests the occurrence of a saddle-node bifurcation,
and indeed the theorem below explains the existence of a curve of saddle-node bifurcations of limit cycles
for the original system \eqref{eq:hybridSystem}.

    \begin{theorem}
        Suppose $F$ and $R$ in \eqref{eq:hybridSystem} are $C^k$ ($k \ge 3$) and $P(u;0,0)$ has a fixed point $\hat{u}$.
        Suppose $\rD P(\hat{u};0,0)$ has an eigenvalue $1$ of algebraic multiplicity one
        and no other eigenvalues with unit modulus.
        Let $w \in \mathbb{R}^{n-1}$ and $v \in \mathbb{R}^{n-1}$ be left and right eigenvectors,
        respectively, for the eigenvalue $1$ with $w^\transpose v=1$ and suppose 
        \begin{align}
            w^\transpose \frac{\partial P}{\partial \eta} &\ne 0, \label{eq:transversalityConditionSN} \\
            w^\transpose \big[(\rD^2 P)(v,v) \big] &\ne 0. \label{eq:nondegeneracyConditionSN}
        \end{align}
        Then there exists $\delta > 0$ and a $C^{k-2}$ function $g : (-\delta,\delta) \to \mathbb{R}$ with $g(0) = 0$
        such that \eqref{eq:hybridSystem} has saddle-node bifurcations of single-impact limit cycles on $\eta = g
        (\mu)$ for all $0 < \mu < \delta$.
        \label{th:hybridSN}
    \end{theorem}

    The codimension-two point $(\mu,\eta) = (0,0)$ is characterized
    by the BEB at $\mu = 0$ and the eigenvalue $1$ associated with the bifurcating limit
    cycle at $\eta = 0$.
     \Cref{th:hybridSN} contains several genericity conditions to ensure saddle-node bifurcations arise.
    The eigenvalue $1$ needs to be a simple eigenvalue and there cannot be any other eigenvalues with unit modulus.
    The {\em transversality condition} \eqref{eq:transversalityConditionSN} ensures $\eta$ unfolds the bifurcation in
    a generic fashion.
    The {\em non-degeneracy condition} \eqref{eq:nondegeneracyConditionSN} ensures $P$ has the required nonlinearity
    to support the creation of two limit cycles.
    These conditions appear in the classical saddle-node bifurcation theorem for maps \cite[pg.~148]{GuHo83}.
    In the proof below we in fact obtain a formula for $g'(0)$ in terms of the derivatives of the restriction of $P$
    to a centre manifold.

\begin{proof}
Since $P$ is smooth (\cref{le:PisSmooth}) classical (smooth) bifurcation theory can be employed.
The map $P$ has a $C^{k-1}$ centre manifold through $\hat{u}$ with local representation $u = h(s;\mu,\eta)$, where $s \in \mathbb{R}$. Let $\hat{s} \in \mathbb{R}$ be such that $\hat{u} = h(\hat{s};0,0)$. Let $p(s;\mu,\eta)$ denote the restriction of $P$ to the centre manifold,
and let $q(z;\mu,\eta)$ denote this map under the translation $z = s - \hat{s}$.
Then $q$ is $C^{k-1}$ and has the form
\begin{equation}
            q(z;\mu,\eta) = z + a \mu + b \eta + c z^2 + d \mu z + e \eta z +  m \mu^2 + n \eta^2 + \text{\sc h.o.t.},
\label{eq:1dSN}
        \end{equation}
where $a,b,c,d,e, m, n \in \mathbb{R}$
and {\sc h.o.t.}~contains all cubic and higher order terms.
Notice the coefficient of the $z$-term in \eqref{eq:1dSN} is
$1$ due to the assumption that $\hat{u}$ has a stability multiplier of $1$.
Following 
\cite[pg.~174]{Yuri98} we have 
\begin{equation}
a = w^\transpose \frac{\partial P}{\partial \mu}, \quad
b = w^\transpose \frac{\partial P}{\partial \eta}, \quad 
c= \frac{1}{2}w^\transpose \big[ (\rD^2 P)(v,v) \big],
\label{eq:adef}
\end{equation}
while formulas for $d$, $e$, $m$ and $n$ are not needed for this proof but are given in  \cref{sec:numerics-the-return-map} as they are important when unfolding the bifurcation (see instead \cref{sec:amp-approx-SN}).
Note that 
$b\neq 0$ 
by \eqref{eq:transversalityConditionSN} and $c \ne 0$ by \eqref{eq:nondegeneracyConditionSN}. Since $b \ne 0$, locally we can solve the fixed point equation $z = q(z;\mu,\eta)$ for $\eta$ giving
        \begin{equation}
            \eta = g_1(\mu,z) = -\frac{a}{b} \,\mu - \frac{c}{b} \,z^2 + \text{\sc h.o.t.},
            \nonumber
        \end{equation}
which is $C^{k-1}$ by the implicit function theorem.
The derivative of $q$ is
        $$
        \frac{\partial q}{\partial z}(z;\mu,\eta) = 1 + 2 c z + \text{\sc h.o.t.},
        $$
which is $C^{k-2}$.
Since $c \ne 0$, locally we can solve $\displaystyle \frac{\partial q}{\partial z}(z;\mu,g_1(\mu,z)) = 1$ for $z$ giving
        $$
        z = g_2(\mu) = \cO(\mu),
        $$
        which is $C^{k-2}$ by the implicit function theorem.
        Thus, saddle-node bifurcations occur on
        \begin{equation}
            \label{eq:SN_gradient}
            \eta = g_1(\mu,g_2(\mu)) = -\frac{a}{b} \,\mu + \cO \left( \mu^2 \right).
        \end{equation}
    \end{proof}

    \subsection{BEB-period-doubling bifurcations}
    \label{sec:pd}

The following result is a natural analogue of the previous theorem for the case of a multiplier $-1$ at $\eta=0$.
The transversality and non-degeneracy conditions are now
\eqref{eq:transversalityConditionPD} and \eqref{eq:nondegeneracyConditionPD}
respectively.
The latter condition appears in the classical period-doubling bifurcation theorem
and ensures a period-doubled solution is created in a generic fashion.
A derivation of \eqref{eq:nondegeneracyConditionPD}
is provided by Kuznetsov \cite[pgs.~184--185]{Yuri98}.
In \Cref{sec:numerics-the-return-map} we provide explicit
formulas for the constants in \eqref{eq:nondegeneracyConditionPD}
in terms of $P$ and its derivatives.

    \begin{theorem}
        Suppose $F$ and $R$ in \eqref{eq:hybridSystem} are $C^k$ ($k \ge 4$) and $P(u;0,0)$ has a fixed point $\hat{u}$.
        Suppose $\rD P(\hat{u};0,0)$ has an eigenvalue $-1$ of algebraic multiplicity one
        and no other eigenvalues with unit modulus.
        Locally the fixed point persists having a $C^{k-1}$ eigenvalue $\lambda(\mu,\eta)$ with $\lambda(0,0) = -1$;
        suppose
        \begin{equation}
            \frac{\partial \lambda}{\partial \eta}(0,0) \ne 0.
            \label{eq:transversalityConditionPD}
        \end{equation}
        and
        \begin{equation}
            c^2 + f \ne 0,
            \label{eq:nondegeneracyConditionPD}
        \end{equation}
        where $c$ and $f$ are the coefficients of quadratic and cubic terms,
        respectively, of the restriction of $P$ to the centre manifold.
        Then there exists $\delta > 0$ and a $C^{k-2}$ function $g : (-\delta,\delta) \to \mathbb{R}$ with $g(0) = 0$
        such that \eqref{eq:hybridSystem} has period-doubling bifurcations
        of single-impact limit cycles on $\eta = g(\mu)$ for all $0 < \mu < \delta$.
        \label{th:hybridPD}
    \end{theorem}

    \begin{proof}
        As in the proof of \cref{th:hybridSN}
        we restrict $P$ to a one-dimensional centre manifold
        and move the critical fixed point to the origin resulting in a map of the form
        \begin{equation}
            q(z;\mu,\eta) = -z + a \mu + b \eta + c z^2 + d \mu z + e \eta z + f z^3 + \text{\sc h.o.t.},
            \label{eq:1dPD}
        \end{equation}
        where $a,b,c,d,e,f \in \mathbb{R}$ and {\sc h.o.t.}~contains all terms
        that are quadratic in $\mu$ and $\eta$ and all cubic and higher order
        terms except the $z^3$-term which has been written explicitly
        because it is important below.
        Let $w \in \mathbb{R}^{n-1}$ and $v \in \mathbb{R}^{n-1}$
        be left and right eigenvectors,
        respectively, for the eigenvalue $-1$ and satisfying $w^\transpose v = 1$.
        Again, $a$, $b$, and $c$ are given explicitly by \eqref{eq:adef}. The coefficients $d$, $e$ and $f$ are slightly harder to compute, as they rely on careful consideration of the parameter dependence of the various co-ordinate transformations. A full derivation is given in  \cref{sec:numerics-the-return-map}. Note that these terms play a role in the unfolding of the dynamics, and ignoring them can alter the criticality of the various bifurcating orbits, and they are important imp making asymptotic predictions for the amplitudes of bifurcating orbits (see \cref{sec:amp-approximation}).
        
        By the implicit function theorem, locally there is a unique $C^{k-1}$ fixed point
        \[
            z = g_1(\mu,\eta) = \frac{a}{2} \,\mu + \frac{b}{2} \,\eta + \cO \left( \left( |\mu| + |\eta| \right)^2
            \right).
            \nonumber
        \]
    The stability multiplier of this fixed point is
        $$
        \lambda(\mu,\eta) = \frac{\partial q}{\partial z}(g_1(\mu,\eta);\mu,\eta)
        = -1 + (a c + d) \mu + (b c + e) \eta
        + \cO \left( \left( |\mu| + |\eta| \right)^2 \right),
        $$
        which is $C^{k-2}$.
        We have $b c + e \ne 0$, by the transversality condition \eqref{eq:transversalityConditionPD},
        so, by the implicit function theorem,
        locally the stability multiplier is $-1$ on a $C^{k-2}$ curve
        \begin{equation}
            \label{eq:PD_gradient}
            \eta = g_2(\mu) = -\frac{a c + d}{b c + e} \,\mu + \cO \left( \mu^2 \right).
        \end{equation}
        For sufficiently small values of $\mu$, period-doubled solutions are created on this curve
        because the non-degeneracy condition \eqref{eq:nondegeneracyConditionPD} is satisfied.
    \end{proof}

\section{Numerical verification}
\label{sec:examples}

In this section we perform bifurcation analyses
on examples to illustrate the unfoldings predicted by
\cref{th:hybridSN,th:hybridPD}.
We evaluate the coefficients in the one-dimensional forms
\eqref{eq:1dSN} and \eqref{eq:1dPD}
and use \eqref{eq:SN_gradient} and \eqref{eq:PD_gradient}
to determine the slope of the saddle-node and period-doubling bifurcation curves
at the codimension-two points,
and compare these to the results of numerical continuation.

\subsection{Numerical computation of the normal form coefficients}
    \label{sec:computation_and_analysis}

We first provide general comments on the manner by which the coefficients
in \eqref{eq:1dSN} and \eqref{eq:1dPD} can be computed.
If the vector field $F$ in \eqref{eq:hybridSystem} is linear,
these coefficients can be computed analytically.
This is because the flow can be expressed explicitly as a matrix exponential,
from which we can find explicit expressions for the Poincar\'{e} map $P$ up to an implicit
transcendental equation for the evolution time $T$. 
Care is needed when evaluating the derivatives of $P$ because
$T$ varies with the initial point
and this needs to be accounted for when differentiating.
Note that even though
$F$ is in this case linear,
$P$ will be nonlinear and satisfies the non-degeneracy conditions
\eqref{eq:nondegeneracyConditionSN} and \eqref{eq:nondegeneracyConditionPD}
in generic situations.
If instead $F$ is nonlinear, $P$ needs to be evaluated numerically
and its derivatives can be computed using linear variational equations
or finite differences.

For linear vector fields $F$ in the examples below, limit cycles in blown-up coordinates were
computed using the semi-analytic continuation method developed in \cite{HoCh23}.
For nonlinear $F$ they were computed from direct numerical simulations of the system. 

In this section codimension-two points are identified as follows.
We vary a second parameter
and solve numerically to find a value of this parameter at which $P$
has a fixed point with stability multiplier $1$ or $-1$ in the limit $\mu \to 0^+$
(flipping the sign of $\mu$ if needed).
This limit is evaluated by working in blown-up coordinates \eqref{eq:blowUp}
and setting $\mu = 0$.
We shift the second parameter so that the codimension-two point occurs at $0$,
and call it $\eta$ so that  \cref{th:hybridSN,th:hybridPD}
can be applied directly.
We then compute the critical eigenvectors 
$w$ and $v$ and the coefficients in \eqref{eq:1dSN} and \eqref{eq:1dPD}
as discussed above.
We can also use formulas obtained in the above proofs
to approximate the amplitude of the bifurcating limit cycles.
Details of these calculations are provided in \Cref{sec:numerics-the-return-map,sec:amp-approximation}.

All computations were carried out in {\sc matlab}
and the code is available at \url{https://github.com/HongTang973}
In what follows we write $\mathbf{e}_i$, where $i = 1,2,3, \cdots$,
for the $i^{\rm th}$ standard basis vector of $\mathbb{R}^n$.

\subsection{Three-dimensional systems}
\label{sec:3d-case-analysis}

We start by considering three-dimensional systems \eqref{eq:hybridSystem}
where the reset law has the form \eqref{eq:R}
and $F$, $W$, and $H$ are truncated from \cref{eq:IHS_series} with the form
\begin{equation}
        \begin{split}
            F(x;\mu,\eta) &= A\eta) x + M \mu + A_1 \mu x + 
            \epsilon (q_{1} x_1 x_2 + q_{2} x_1 x_3) \mathbf{e}_1 \,, \\
            W(x;\mu,\eta) &= -B, \\
            H(x;\mu,\eta) &= C^\transpose x,
        \end{split}
        \label{eq:codim2_NL_Nform}
\end{equation}
where $A , A_1\in \mathbb{R}^{3 \times 3}$,
$B, C, M \in \mathbb{R}^3$,
and $q_{1}, q_{2}, \epsilon \in \mathbb{R}$.
We first consider cases where $F$ is linear, i.e.~$\epsilon = 0$,
then show that the same phenomenon occurs when $F$ is nonlinear.
The term $A_1 \mu x$ (higher order in \cref{eq:IHS_series})
provides additional $\mu$-dependency.

\subsubsection{BEB-saddle-node bifurcations}
\label{subsec:beb-collides-with-the-sn-bifurcation}

\begin{Example}
\label{itm:3D_SN_case}
Suppose the coefficients in \eqref{eq:codim2_NL_Nform}
have the form
\begin{equation}
A = \begin{bmatrix} t & 1 & 0 \\ m & 0 & 1 \\ d & 0 & 0 \end{bmatrix}, \quad
B = \begin{bmatrix} 0 \\ b_2 \\ b_3 \end{bmatrix}, \quad
C = \mathbf{e}_1, \quad
M = -\mathbf{e}_3, \quad
A_1 = -\mathbf{e}_1^{\transpose} \mathbf{e}_1,
\label{eq:3D_SN_case}
\end{equation}
where $t, m, d, b_2, b_3 \in \mathbb{R}$.
Since $A$ is a companion matrix
(cf.~\cite{HoCh23}),
its parameters are given in terms of its eigenvalues
$\lambda_1$, $\lambda_2$, and $\lambda_3$ by
$$
t = \lambda_1 + \lambda_2 + \lambda_3 \,, \quad
m = -\lambda_1 \lambda_2 - \lambda_1 \lambda_3 - \lambda_2 \lambda_3 \,, \quad
d = \lambda_1 \lambda_2 \lambda_3 \,.
$$
We set the eigenvalues as
$$
\lambda_1 = -0.1 + 0.2 {\rm i}, \quad
\lambda_2 = -0.1 - 0.2 {\rm i}, \quad
\lambda_3 = -0.5.
$$
and fix $b_3 = 1.6$ and $\epsilon = 0$
and use $b_2$ as the second bifurcation parameter.
\end{Example}

The key quantities in \cref{th:CAB} are 
$$
C^{\transpose}AB = b_2, \qquad
C^{\transpose}A^{-1}B = \frac{b_3}{d}, \qquad
C^{\transpose}A^{-1}M = -\frac{1}{d},
$$
so with $b_2 > 0$ the BEB at $\mu = 0$ corresponds to persistence.

Numerically we found that $P$ in the limit $\mu \to 0^+$
has a fixed point with stability multiplier $1$
when $b_2 = b_2^0 \approx 1.7819$, and 
the assumptions of \cref{th:hybridSN}
hold using $\eta = b_2 - b_2^0$.
 for $\mu > 0$saddle-node bifurcations occur on
Thus a curve $b_2 = b_2^0 + g(\mu)$ with $g(0) = 0$.
We computed this curve numerically and it is shown in \Cref{fig:sn_codim2_curve_a}.
Also as a dashed line we show the theoretical linear approximation
$b_2 = b_2^0 - \frac{a}{b} \mu$, due to \cref{eq:SN_gradient},
where $a \approx 0.1695$ and $b \approx -0.0391$ for this example.
As expected, the slope $ \displaystyle -\frac{a}{b}$ appears to match that of the numerically computed curve
at the codimension-two point.

\begin{figure}[ht!]
	\centering
	\subfloat[ ]
		{
		\includegraphics[width = 0.45 \linewidth]{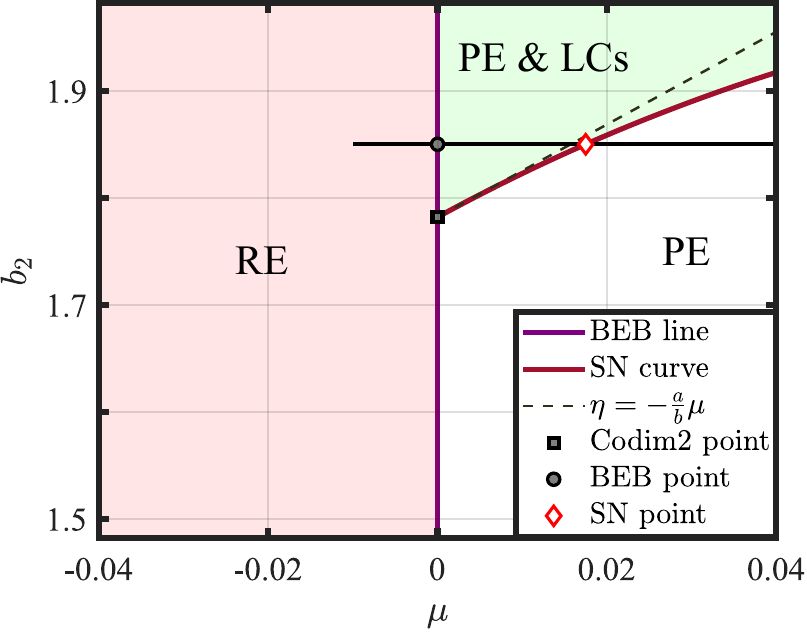}
		\label{fig:sn_codim2_curve_a}
	}
	\subfloat[ ]
	{
		\includegraphics[width = 0.45 \linewidth]{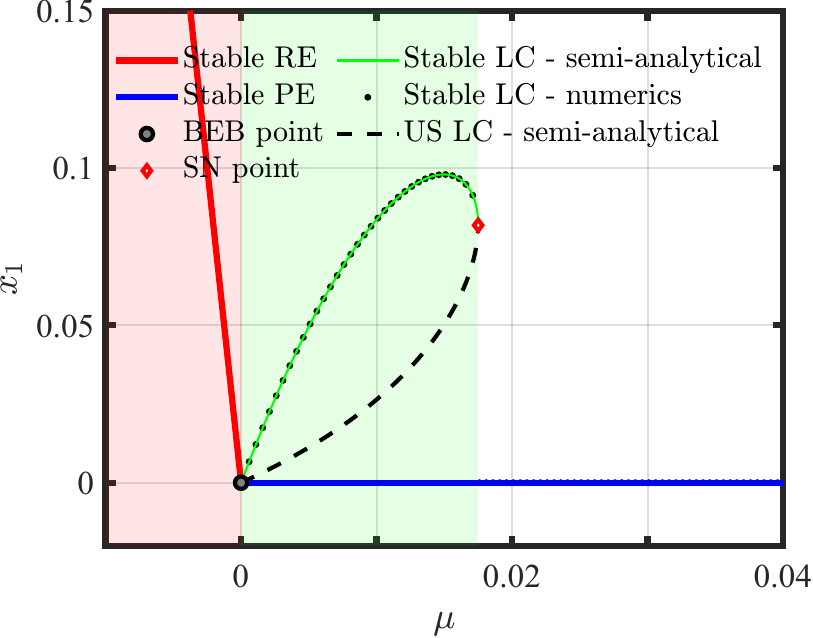}
		\label{fig:sn_codim2_curve_b}
	}
	\caption{
        (a) A two-parameter bifurcation diagram for \cref{itm:3D_SN_case}
        (RE: regular equilibrium; PE: pseudo-equilibrium; LC: limit cycle;
        SN: saddle-node bifurcation; BEB: boundary equilibrium bifurcation).
        (b) A one-parameter slice along the horizontal line at $b_2 = 1.85$ shown in panel (a).
    }
	\label{fig:codim2_SN_curve}
\end{figure}

The BEB and saddle-node bifurcation curves divide
the parameter plane locally into three regions.
For $\mu < 0$ the system has a stable regular equilibrium (RE),
for $\mu > 0$ below the saddle-node curve it has a stable pseudo-equilibrium (PE),
while for $\mu > 0$ above the saddle-node curve it has a stable pseudo-equilibrium
and stable and unstable single-impact limit cycles (LCs).

\begin{figure}[h!]
	\centering
	\subfloat[]{
		\includegraphics[width = 0.45 \linewidth]{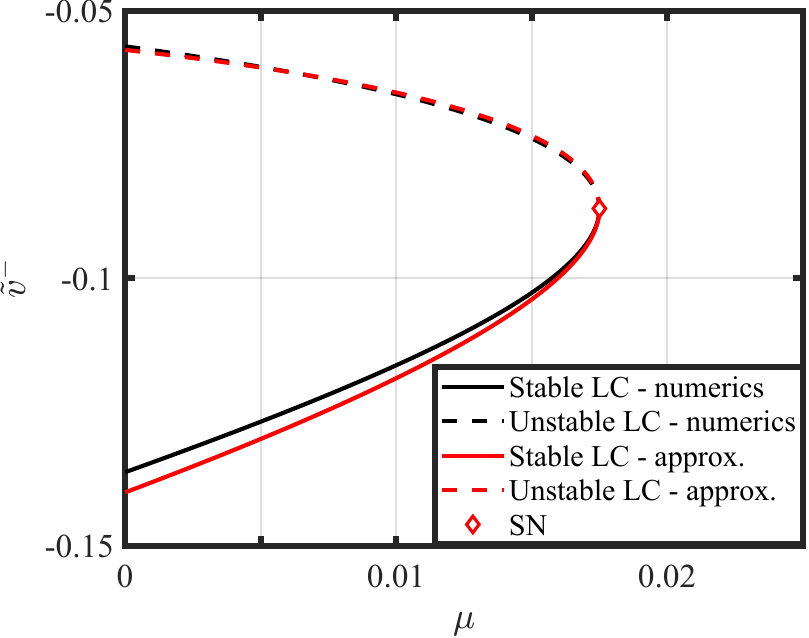}
		\label{fig:sn_codim2_unfold_a}
	}
	\subfloat[]{
		\includegraphics[width = 0.45 \linewidth]{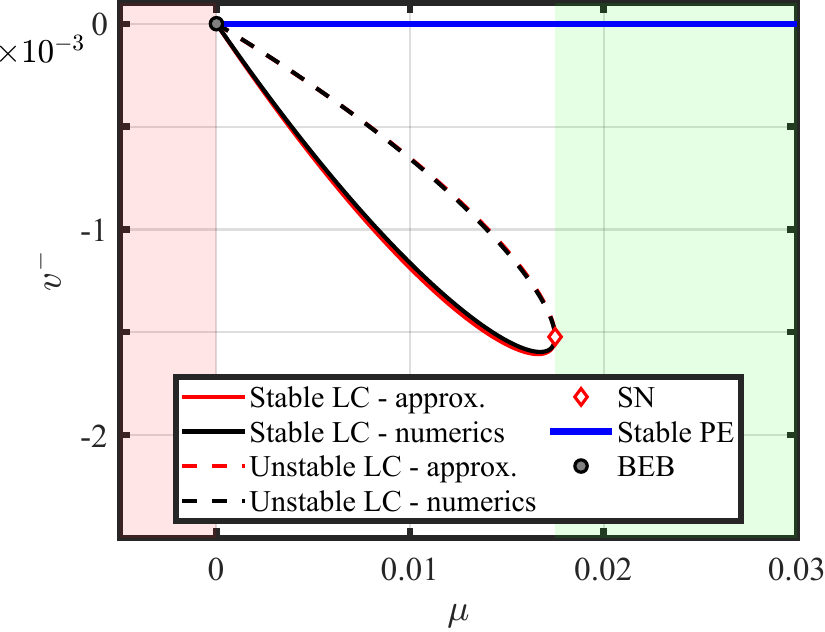 }
		\label{fig:sn_codim2_unfold_b}
	}
	\caption{(a) A comparison of the impact velocities of the limit cycles
    in blown-up co-ordinates computed numerically (black) and via the normal-form approximation (red), using $b_2=1.85$ as in
    \Cref{fig:sn_codim2_curve_b}. (b) The same quantities in original co-ordinates;  the blue line denotes the stable pseudo-equilibrium.}
	\label{fig:sn_amplitude_approximation}
\end{figure}

\Cref{fig:sn_codim2_curve_b} shows a bifurcation diagram for a one-parameter slice through
both bifurcation curves defined by fixing $b_2 = 1.85$.
This shows the $x_1$-values of the regular and pseudo equilibria
and the maximum $x_1$-values attained by the stable and unstable limit cycles.
The limit cycles grow asymptotically linearly from the BEB at $\mu = 0$,
then collide and annihilate in the saddle-node bifurcation at
$\mu_{\rm SN} = g^{-1} \left( b_2 - b_2^0 \right) \approx 0.0175$.
To produce \Cref{fig:sn_codim2_curve_b} the stable limit cycle was continued numerically,
while both stable and unstable limit cycles were continued semi-analytically via the method in \cite{HoCh23}
(one point on the limit cycle was obtained by solving an eigenproblem, then
numerical integration carried out to obtain the maximum $x_1$-value of the limit cycle).

Next, we examine the impact velocities of the limit cycles.
These are shown in \Cref{fig:sn_codim2_unfold_a}
in blown-up coordinates and in \Cref{fig:sn_codim2_unfold_b}
in original coordinates.
That is, \Cref{fig:sn_codim2_unfold_b}
shows the velocity $v^- = \nabla H(x)^{\sf T} F(x)$
evaluated at the point $x$ at which the limit cycle returns to the
switching surface, and \Cref{fig:sn_codim2_unfold_a}
shows $\tilde{v}^- = \frac{v^-}{\mu}$.
As shown in \cref{sec:amp-approximation},
the impact velocities are well approximated by
\begin{equation}
\mathcal{A} = \mu \left( \ell_0  + \ell_1 \mu \pm L_{10}\sqrt{-\frac{a (\mu - \mu_{\rm SN})}{c}  + \frac{(d^2 - 4mc)(\mu - \mu_{\rm SN})^2}{4c^2} } - \frac{d(\mu - \mu_{\rm SN})}{2c}\right),
\end{equation}
where $a$, $c$, $d$ and $m$
are coefficients in the map \eqref{eq:1dSN} on the centre manifold,
and $\ell_0$ and $\ell_1$ are appropriate coefficients
for converting the position on the centre manifold
to the impact velocity.
We obtained
$a \approx 0.1695$, $b \approx -0.0391$,
$c \approx -1.3833$, $d \approx -5.2939$, $e \approx -0.0737$, $m \approx 4.5852$,
$\ell_0 \approx -0.0870$, $\ell_1 = 0 $ and $L_{10}  \approx0.1544$,
by numerically evaluating derivatives of $P$ at the
codimension-two point $(\mu,\eta) = (0,0)$
The resulting approximations to the impact velocities
are shown in \Cref{fig:sn_amplitude_approximation}
and seen to match well to those obtained via direct integration.
A better match can be achieved by evaluating the coefficients at
the saddle-node bifurcation $(\mu,\eta) = (\mu_{\rm SN}, 1.85)$.

\begin{Example}
    \label{itm:3D_SN_case_NL}
    Suppose the coefficients are as in \cref{itm:3D_SN_case}, except
    \qquad    $$ (a). \quad 
    q_{1} = -1, ~
    q_{2} = 0;
     \qquad
     (b). \quad
    q_{1} = 0, ~
    q_{2} = 1;
    $$
    and we allow various values of $\epsilon \ge 0$.
    \end{Example}
 expansion
Now with $\epsilon \ne 0$ the vector field $F$ is nonlinear.
The location of the codimension-two point
in unchanged from \cref{itm:3D_SN_case}
because the nonlinear terms contribute order-$\mu$
perturbations to the system in blown-up co-ordinates.
The nonlinear terms do affect the slope $-\frac{a}{b}$ of the curve
of saddle-node bifurcations.
This explains \cref{fig:nl_BEB_SN_diagram} which shows that
for small $\eta = b_2 - b_2^0 > 0$ two limit cycles are created
in the BEB at $\mu = 0$, but the location of the saddle-node bifurcation
where these are destroyed depends on the value of $\epsilon$.

\begin{figure}[h!]
	\centering
	\subfloat{
		\includegraphics[width = 0.45 \linewidth]{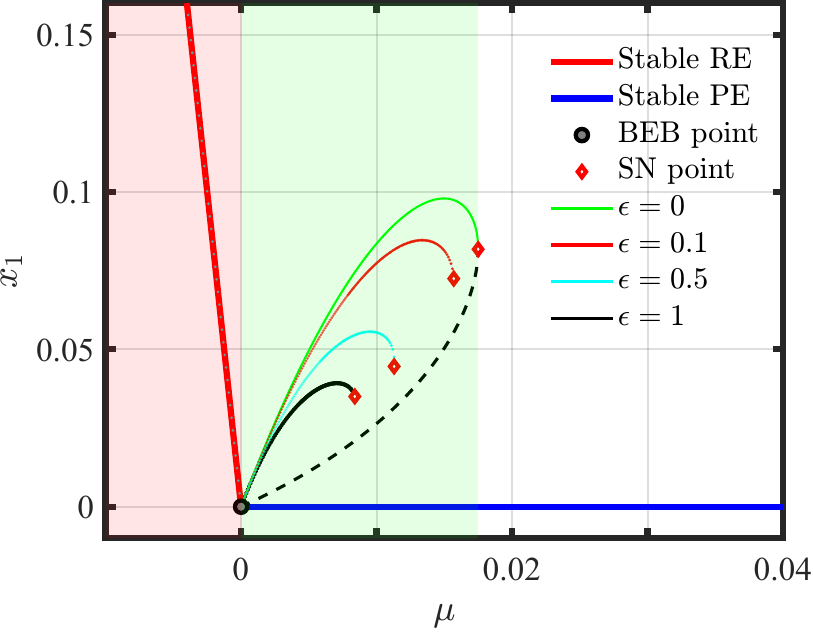}
		\label{fig:sn_codim2_unfold_a_nl}
	}
	\subfloat{
		\includegraphics[width = 0.45 \linewidth]{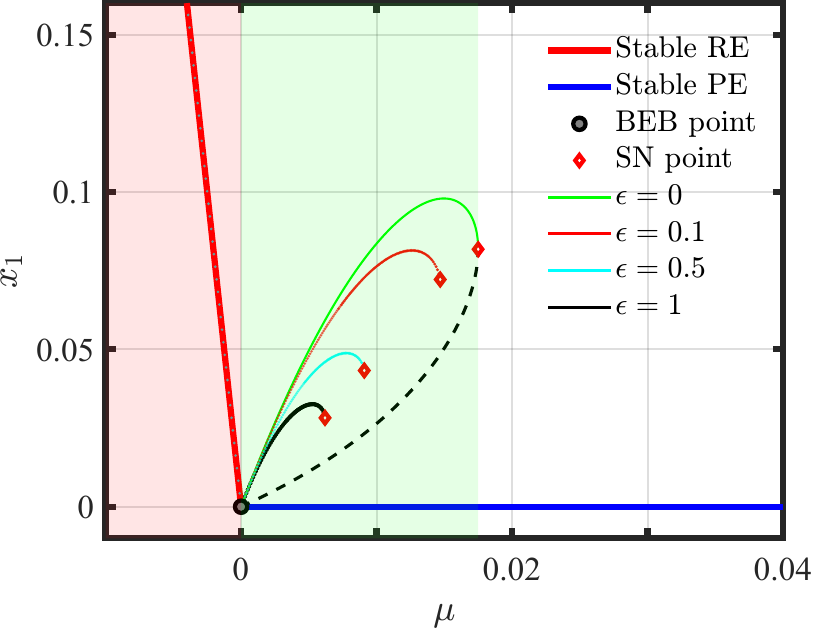}
		\label{fig:sn_codim2_unfold_b_nl}
	}
	\caption{One-parameter bifurcation diagrams for \cref{itm:3D_SN_case_NL}
    with $b_2 = 1.85$ and various values of $\epsilon \ge 0$.
    Panel (a) uses $q_1 = -1$ and $q_2 = 0$,
    while panel (b) uses $q_1 = 0$ and $q_2 = 1$.
 }
	\label{fig:nl_BEB_SN_diagram}
\end{figure}

\subsubsection{BEB-period-doubling bifurcations}
\label{subsec:beb-collides-with-pd}

\begin{Example} \label{itm:3D_PD_case}
    Inspired by forthcoming research into Shilnikov chaos arising from a codimension-three BEB \cite{Shil24},
    suppose the coefficients in \eqref{eq:codim2_NL_Nform} have the form
	$$
    A =
	\begin{bmatrix}
		\rho  & \omega & 0 \\ -\omega & \rho & 1 \\ 0 & 0 & -\lambda
	\end{bmatrix}, \quad
    A_1 = \mathbf{e}_3^\transpose \mathbf{e}_3, \quad
	M = \mathbf{e}_3, \quad
    C = \mathbf{e}_1, \quad
    B = [0, 1+r , -\sigma ]^\transpose.
	$$
    where $\sigma$ is a new parameter unrelated to that in \cref{eq:hybridSystemScaledTruncated}.
    We fix $\epsilon = 0$ and
\begin{equation}
\rho = 0.1, \qquad \omega = 1, \qquad \lambda = 0.3,
\label{eq:rho_omega_lambda}
\end{equation}
and consider various values of $r$ and $\sigma$.
	\end{Example}

For this example
$$
C^{\transpose}AB = (1 + r)\omega, \quad
C^{\transpose}A^{-1}B = \frac{\omega(\sigma-\lambda(1+r)) }{\lambda(\omega^2 + \rho^2)}, \quad
C^{\transpose} A^{-1} M = \frac{-\omega}{\lambda (\omega^2 + \rho^2)}.
$$
With $\sigma = \sigma_0 = 0.8$
we found numerically that $P$ in the limit $\mu \to 0^+$
has a fixed point with stability multiplier $-1$ when $r = r_0 \approx 0.66691$.
Moreover, the assumptions of
\cref{th:hybridPD} hold
using $\eta = \sigma - \sigma_0$ for the second parameter.
Thus for $\mu > 0$ period-doubling bifurcations occur on a curve
$\sigma = \sigma_0 + g(\mu)$ with $g(0) = 0$.
We computed this curve numerically and it is shown in \Cref{fig:PD_codim2_curve_a}.
By \cref{eq:PD_gradient} the curve
emanates from the codimension-two point $(\mu,\sigma) = (0,\sigma_0)$
with slope $-\frac{a c + d}{b c + e}$.
For this example
$a \approx 2.6886$, $b \approx -2.0195$, $c \approx 1.0021$,
$d \approx 0.9756$, and $e \approx -5.2383$ (also $f \approx 0.6332$).
In \Cref{fig:PD_codim2_curve_a} we show
the line through $(\mu,\sigma) = (0,\sigma_0)$
with slope $-\frac{a c + d}{b c + e}$ as a dashed line,
and indeed this appears to match the slope of our numerically computed
period-doubling bifurcation curve.
($\mu =0, ~\eta =0$) 

The BEB and period-doubling bifurcation curves divide
the parameter plane locally into three regions.
For $\mu < 0$ the system has an unstable pseudo-equilibrium,
while for $\mu > 0$ it has an unstable regular equilibrium
and a single-impact limit cycle.
Below the period-doubling curve this limit cycle is stable,
whereas immediately above the period-doubling curve it is unstable
and there is a stable two-impact limit cycle.
Further above the period-doubling curve
stable multi-impact limit cycles exist
due to the occurrence of a period-doubling cascade, described below.


%
\begin{figure}[ht!]
	\centering
	\subfloat[]
		{
		\includegraphics[width = 0.45 \linewidth]{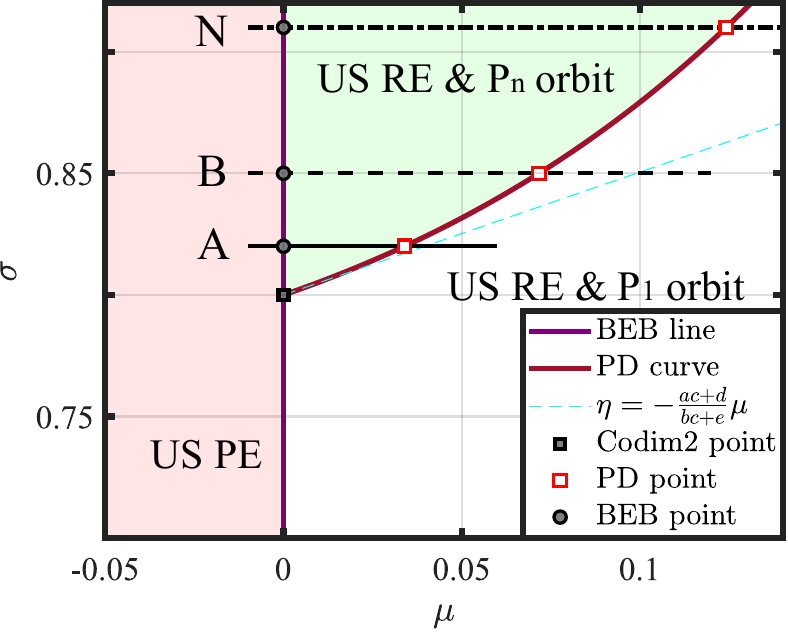}
		\label{fig:PD_codim2_curve_a}
	}
	\subfloat[]
	{
		\includegraphics[width = 0.45 \linewidth]{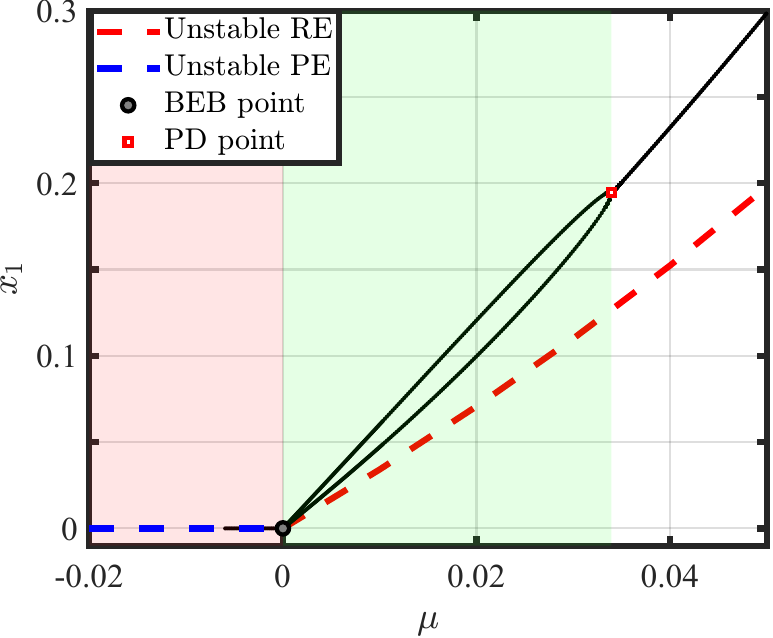}
		\label{fig:PD_codim2_curve_b}
	}
	\caption{(a) A two-parameter bifurcation diagram for \Cref{itm:3D_PD_case}
    with $r = r_0 \approx 0.66691$
    as discussed in the text.
    (b) A one-parameter bifurcation diagram
    using also $\sigma = 0.82$.
    This corresponds to the line $A$ shown in panel (a);
    slices along the lines $B$ and $N$ are shown in \Cref{fig:codim2_PD_chaos}.
    }
	\label{fig:codim2_PD_curve}
\end{figure}

Next we describe several bifurcation diagrams
defined by fixing $\sigma > \sigma_0$.
First \Cref{fig:PD_codim2_curve_b} uses $\sigma = 0.82$ (i.e.~$\eta = 0.02$).
This shows a stable two-impact limit cycle being created in the
BEB at $\mu = 0$, and coexisting with the regular equilibrium.
The limit cycle exists up until the period-doubling bifurcation
at $\mu_{\rm PD} = \eta^{-1}(\sigma - \sigma_0)\approx 0.033941$
beyond which there is a stable single-impact limit cycle.

\begin{figure}[ht!]
	\centering
	\subfloat{
		\includegraphics[width = 0.45 \linewidth]{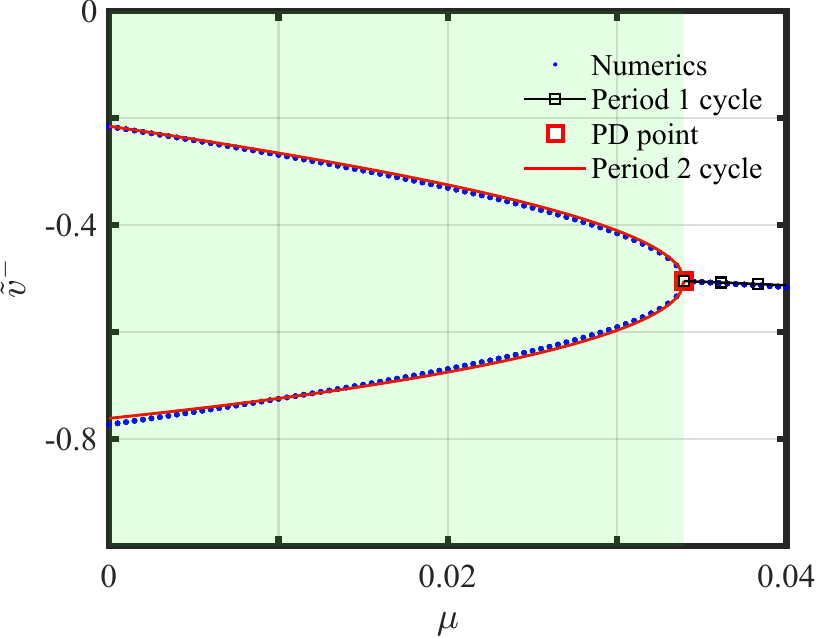}
		\label{fig:PD_codim2_unfold_a}
	}
	\subfloat{
		\includegraphics[width = 0.45 \linewidth]{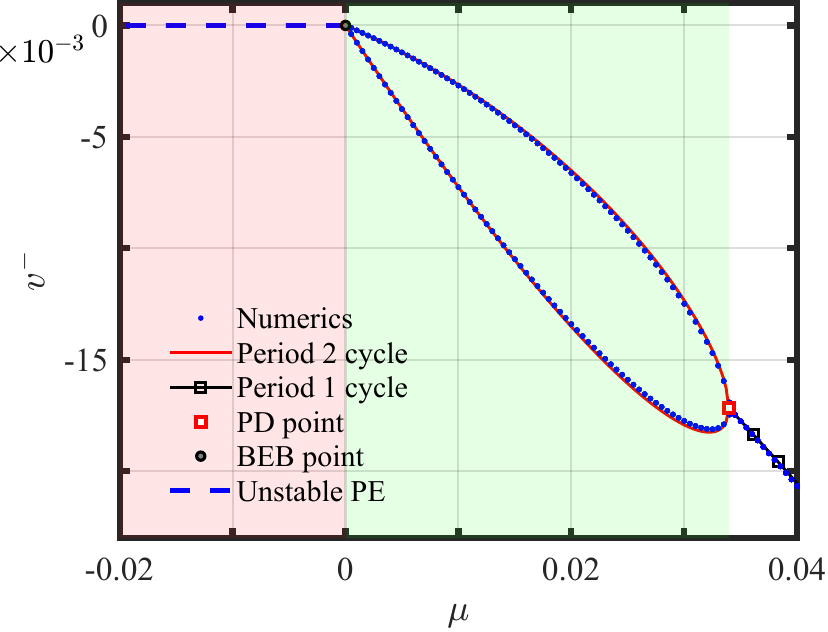}
		\label{fig:PD_codim2_unfold_b}
	}
	\caption{
    The impact velocities of the single-impact and two-impact
    limit cycles of \Cref{fig:PD_codim2_curve_b}
    in blown-up co-ordinates (panel (a)), and original co-ordinates (panel (b)).
    The blue points are the numerically obtained impact velocities,
    while the curves are the theoretical approximations \eqref{eq:sI}
    and \eqref{eq:dI}.
    Similar to \Cref{fig:sn_amplitude_approximation} but showing the agreement between theory and numerics for the BEB-PD case in \cref{itm:3D_PD_case}.}
	\label{fig:amplitude_approximation_PD}
\end{figure}
\begin{figure}[ht!]
	\centering
	\subfloat[]
	{
		\includegraphics[width = 0.44 \linewidth]{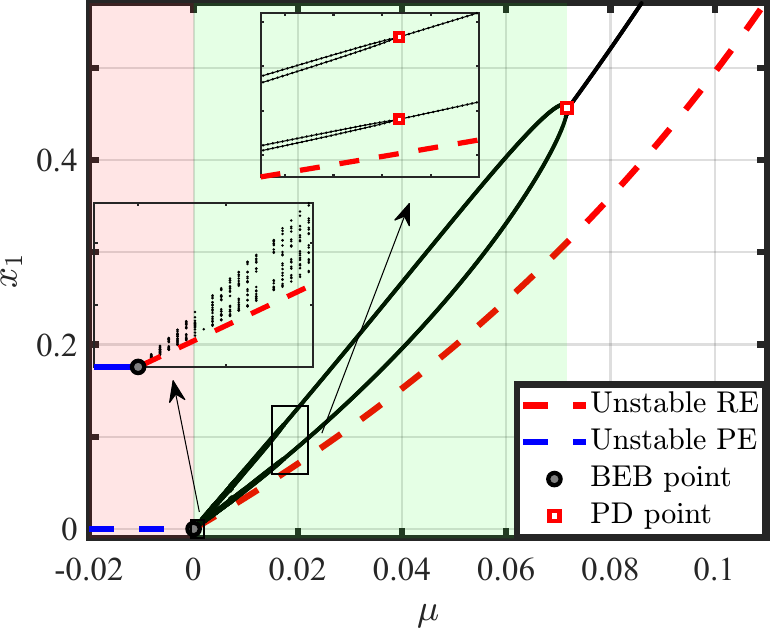}
		\label{fig:PD_codim2_unfold_P3_full}
	}
	\subfloat[]
	{
		\includegraphics[width = 0.45 \linewidth]{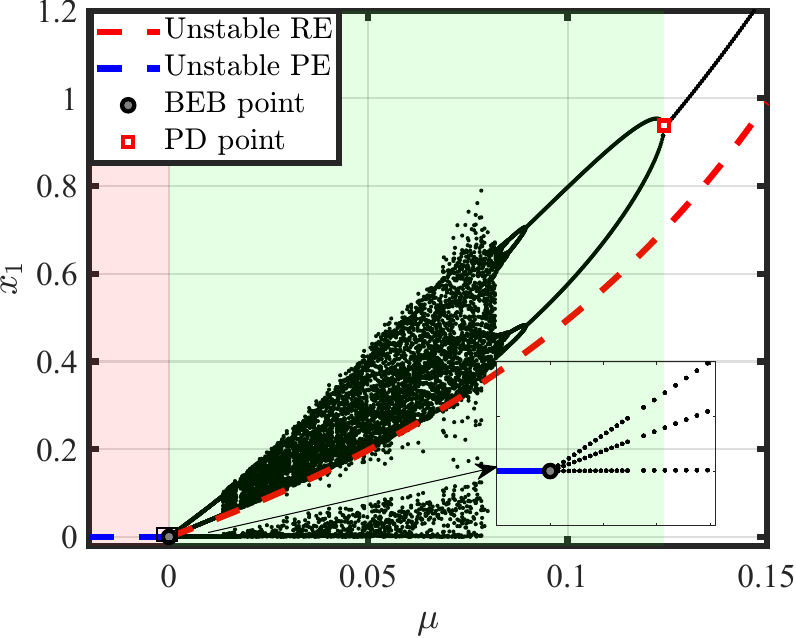}
		\label{fig:PD_codim2_unfold_PN_full}
	}
 \\
 \subfloat[]{
		\includegraphics[width = 0.44 \linewidth]{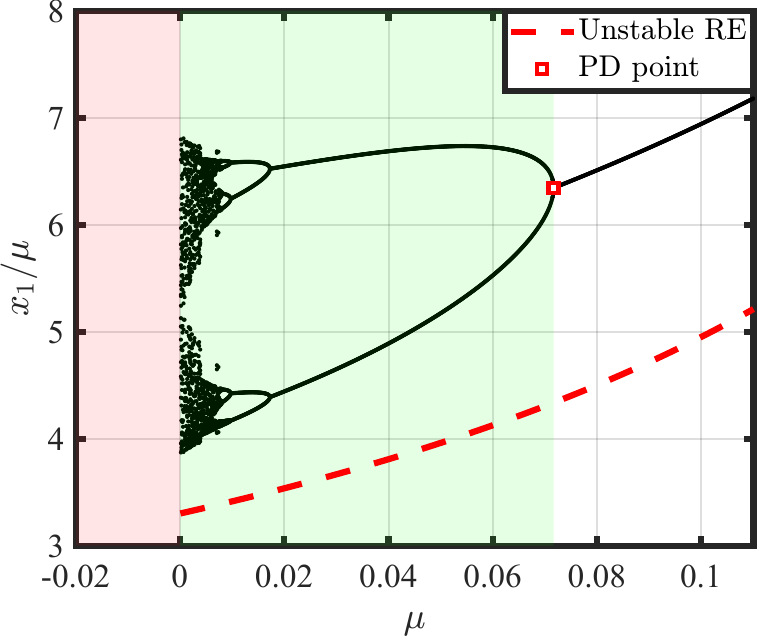}
		\label{fig:PD_codim2_unfold_P3_bu}
	}
	\subfloat[]{
		\includegraphics[width = 0.45\linewidth]{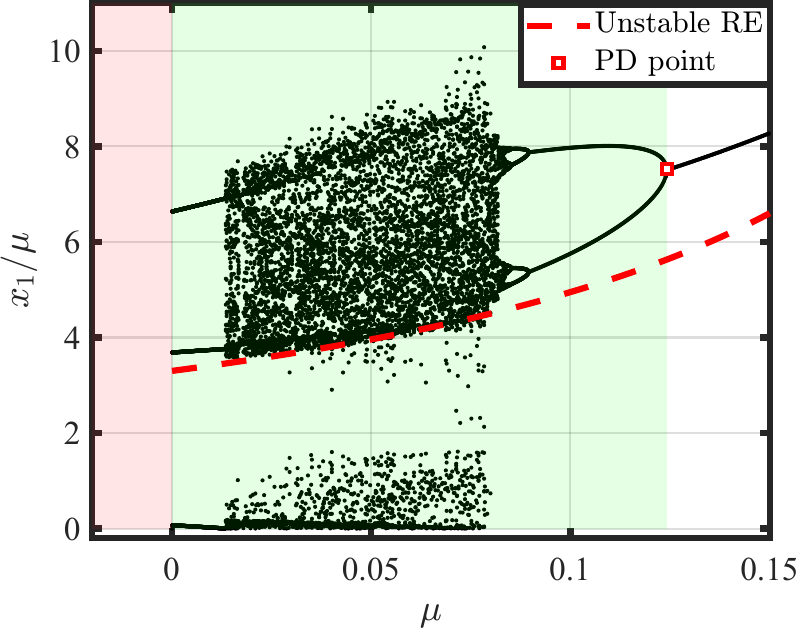}
		\label{fig:PD_codim2_unfold_PN_bu}
	}
    \caption{
    Bifurcation diagrams for \Cref{itm:3D_PD_case}
    with $r = r_0 \approx 0.66691$.
    Panels (a) and (c) use $\sigma = 0.85$ corresponding to slice B in
    \Cref{fig:PD_codim2_curve_a}, while
    panels (b) and (d) use $\sigma = 0.91$ corresponding to slice N.
    The top panels use original co-ordinates,
    while the bottom panels use blown-up co-ordinates.
    }
    \label{fig:codim2_PD_chaos}
\end{figure}
\begin{figure}[h!]
	\centering
	\subfloat[]{
		\includegraphics[width = 0.45 \linewidth]{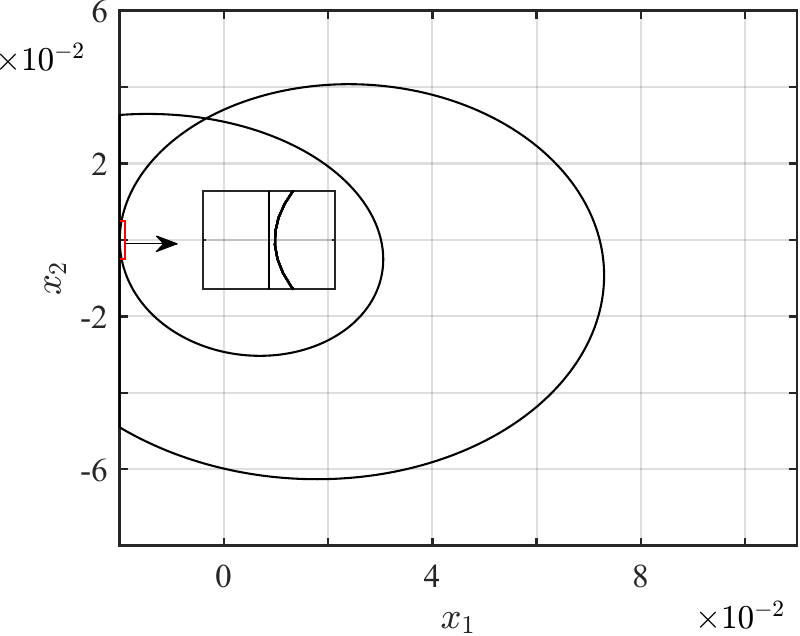}
	}
	\subfloat[]{
		\includegraphics[width = 0.45 \linewidth]{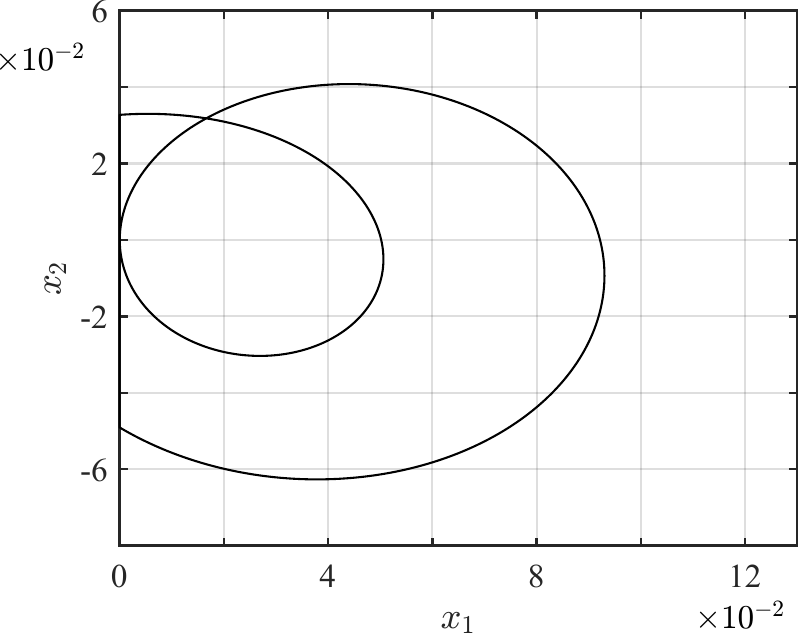}
	}
 \\
        \subfloat[]{
		\includegraphics[width = 0.45 \linewidth]{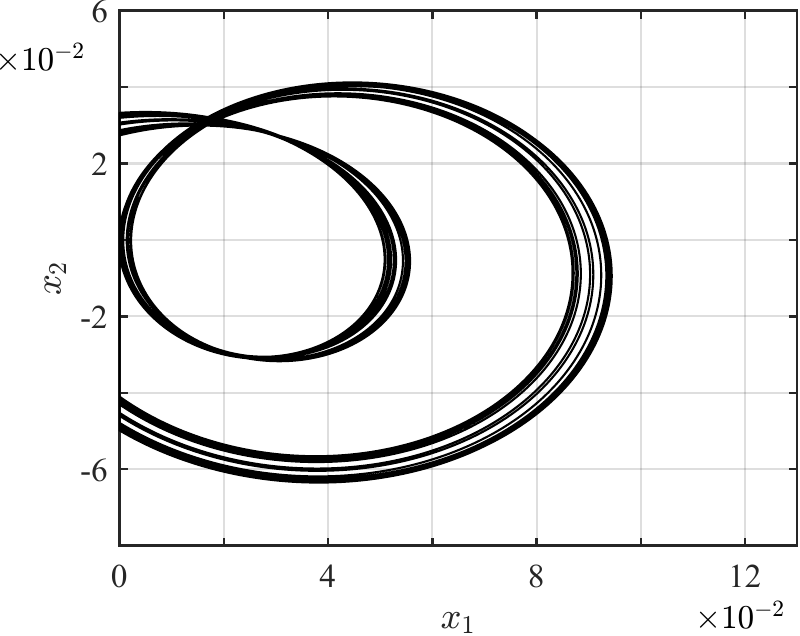}
	}
	\subfloat[]{
		\includegraphics[width = 0.45 \linewidth]{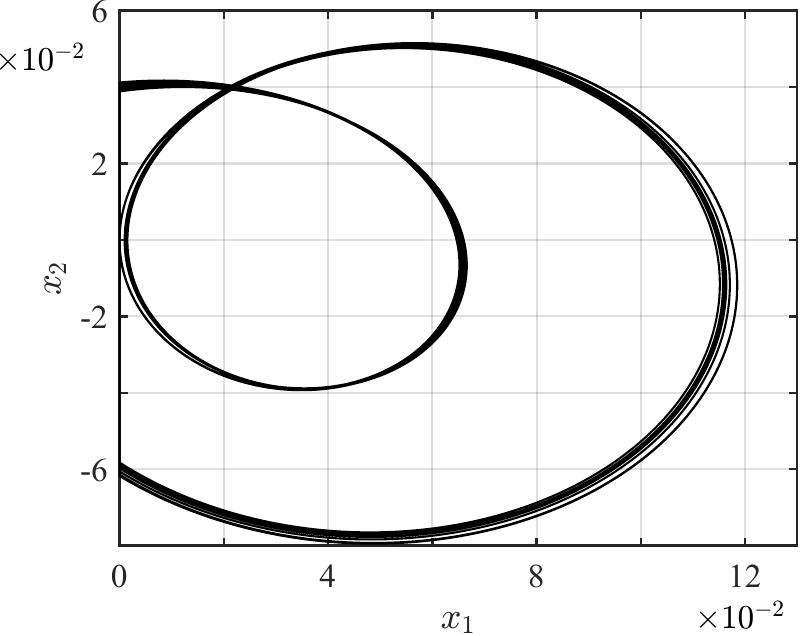}
	}
	\caption{
    Phase portraits showing the attractor of \Cref{itm:3D_PD_case}
    with $r = r_0 \approx 0.66691$, $\sigma = 0.91$,
    and four different values of $\mu$:
    (a) $\mu=0.01345$; (b) $\mu=0.0134597$; (c) $\mu=0.01355$; (d) $\mu=0.017$.
    As the value of $\mu$ is increased,
    a stable limit cycle undergoes a grazing bifurcation at $\mu \approx 0.0134597$
    and appears to be replaced by a chaotic attractor.
    }
\label{fig:slice_C_phase_portraits}
\end{figure}

As shown in \cref{sec:amp-approximation},
the impact velocity in blown-up co-ordinates
of the single-impact limit cycle is well approximated by
\begin{equation}
\hat{\mathcal{A}} = \ell_0 + \ell_1 \mu + L_{01} \frac{a(\mu - \mu_{\rm PD})}{2},
\label{eq:sI}
\end{equation}
where $\ell_0 \approx -0.5053$, $\ell_1 =0$, and $L_{01} \approx -0.9899$
for this example.
Furthermore
\begin{equation}
give\hat{\mathcal{A}} = \ell_0 + \ell_1 \mu + L_{01}\left(\pm \sqrt{ -\frac{(ca+d)(\mu - \mu_{\rm PD})}{c^2+f}} + \frac{c(ca+d)}{ 2(c^2 + f)}(\mu - \mu_{\rm PD}) + \frac{a(\mu - \mu_{\rm PD})}{2} \right)
\label{eq:dI}
\end{equation}
approximates the two impact velocities of the two-impact limit cycle.
As shown in \Cref{fig:PD_codim2_unfold_a}
these provide excellent agreement to the numerically computed values,
and \Cref{fig:PD_codim2_unfold_b}
shows these in the original unscaled co-ordinates.

Next, the left column of \Cref{fig:codim2_PD_chaos}
uses $\sigma = 0.85$, corresponding to slice B in \Cref{fig:PD_codim2_curve_a}.
Specifically \Cref{fig:PD_codim2_unfold_P3_full}
uses original co-ordinates and
\Cref{fig:PD_codim2_unfold_P3_bu} uses blown-up co-ordinates.
As we move right to left across either plot,
the single-impact limit cycle loses stability in a period-doubling bifurcation,
as in \Cref{fig:PD_codim2_curve_b},
but now this appears to be followed by a full period-doubling cascade
resulting in a chaotic attractor.
As we instead move left to right,
this attractor is born in the BEB at $\mu = 0$.
This suggests that as we increase the value of $\sigma$ from $0.82$ to $0.85$,
we pass through infinitely many period-doubling bifurcation curves.
We expect each of these emanates transversally
from the BEB curve at some $0.82 < \sigma < 0.85$,
and this could be verified formally by applying \cref{th:hybridPD}
to higher iterates of $P$.


Finally, the right column of \Cref{fig:codim2_PD_chaos}
uses $\sigma = 0.91$, corresponding to slice N in \Cref{fig:PD_codim2_curve_a}.
Now the BEB creates a stable period-two limit cycle.
At $\mu \approx 0.0134597$ the limit cycle appears
to be converted to a chaotic attractor
through a grazing bifurcation (i.e.~zero velocity impacts), see \Cref{fig:slice_C_phase_portraits}.
The chaotic attractor appears to extend until $\mu \approx 0.075$ where it undergoes a kind of boundary crisis caused by grazing behaviour of the chaotic attractor.
A different kind of chaotic attractor now occurs,
and this is subsequently destroyed in a reverse period-doubling cascade. 

\begin{Example}
    \label{itm:3D_PD_case_NL}
    Suppose the coefficients are as in 
    All the parameters are the same as in \cref{itm:3D_PD_case}, except 
    \qquad    $$ (a). \quad 
    q_{1} = -1, ~
    q_{2} = 0 ~
    ;
     \qquad
     (b). \quad
    q_{1} = 0, ~
    q_{2} = -1 ~
    ;
    $$
    and we allow various values of $\epsilon \ge 0$.
    
\end{Example}

We now consider $\epsilon \ne 0$ to investigate the effect of nonlinear terms.
We return to slice A of \Cref{fig:PD_codim2_curve_b}
where a stable two-impact limit cycle is created at the BEB,
and this occurs for any value of $\epsilon$.
Different values of $\epsilon$ affect the dynamics
and bifurcations occurring for $\mu > 0$, see \Cref{fig:nl_BEB_PD_diagram}.
With $\epsilon = 0.1$ the bifurcation structure is qualitatively unchanged
in that the two-impact limit cycle reverts to a single-impact limit cycle
in a (reverse) period-doubling bifurcation,
but for larger values of $\epsilon$ the attractor
appears to be chaotic in some places.

%
\begin{figure}[h!]
	\centering
	\subfloat{
		\includegraphics[width = 0.45 \linewidth]{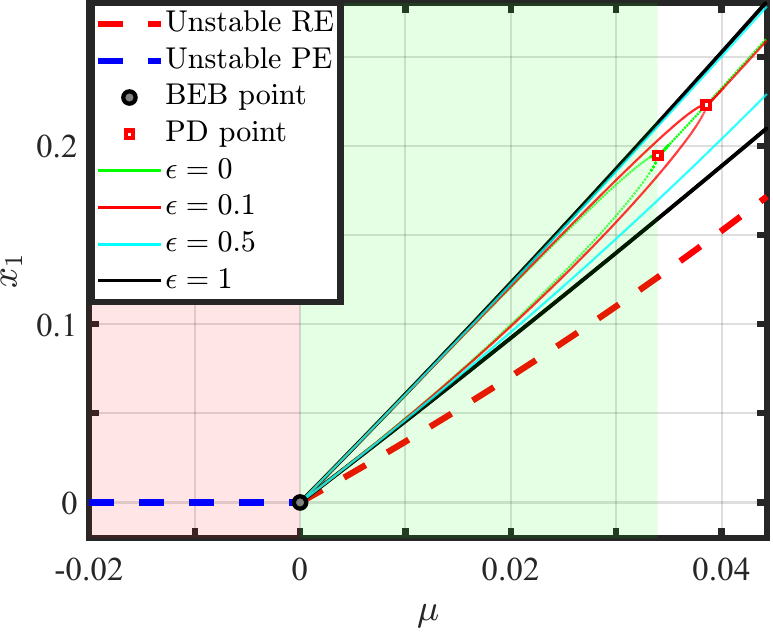}
		\label{fig:pd_codim2_unfold_a_nl}
	}
	\subfloat{
		\includegraphics[width = 0.45 \linewidth]{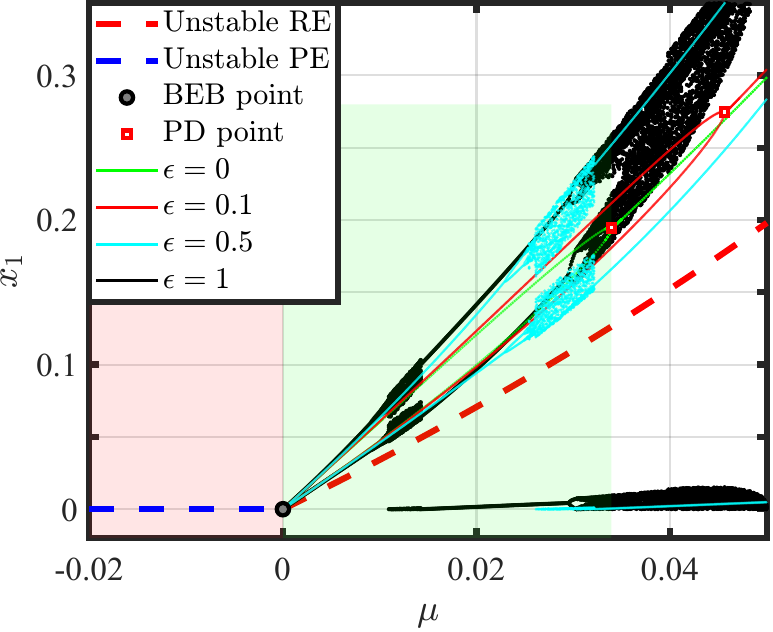}
		\label{fig:pd_codim2_unfold_b_nl}
	}
	\caption{One-parameter bifurcation diagrams for \cref{itm:3D_PD_case_NL}
    with $r = r_0 \approx 0.66691$, $\sigma = 0.82$,
    and various values of $\epsilon$.
    Panel (a) uses $q_1 = -1$ and $q_2 = 0$,
    while panel (b) uses $q_1 = 0$ and $q_2 = -1$.
 }
	\label{fig:nl_BEB_PD_diagram}
\end{figure}

\subsection{An eight-dimensional wing-flap model}
\label{sec:wingflaps}
Here we replicate numerical results from \cite{PeterAIAA} for a simplified airfoil model
illustrated in \Cref{fig: airfoil model}.
Due to rotary freeplay in the hinge between the flap and main body,
the system is well modelled as an impacting hybrid system,
where a reset map is applied when the flap hits the stop.
Using Lagrangian mechanics and approximated aerodynamics, see details in \cite{HoCh23}, the governing equations of motion can be written as
\begin{equation}
	\begin{cases}
		\dot{\mathbf{x}} = \mathbf{A}_{\rm af}(\bar{U}) \mathbf{x} + \mathbf{G}(\bar{U}), \quad \mbox{for} \quad
		\vert \beta \vert < \delta, \\
		\mathbf{x} = R (\mathbf{x}), \quad \mbox{for} ~{\tiny }\vert \beta \vert = \delta,
	\end{cases}
	\label{eq: airfoil model}
\end{equation}
where $\mathbf{x} = [\zeta, \alpha, \beta, \dot{\zeta}, \dot{\alpha}, \dot{\beta}, w_1, w_2]^\transpose$.
Here $\alpha$ is the rotary pitch,
$\beta$ is the flap position,
$\zeta$ is the dimensionless heave,
and $w_1$ and $w_2$ are augmented variables that capture
Theodorsen aerodynamic interactions \cite{THEODORSEN1935}.
The parameter $\bar{U}$ is the dimensionless air velocity over linear flutter velocity $U_{\rm f} = 30 ~{\rm m/s}$,
and $\delta$ characterizes the amount of flap freeplay.
The matrix $\mathbf{A}_{\rm af}$ specifies the dynamics of the airfoil when the flap is in freeplay.
The reset map $R(\mathbf{x})$ is
affine and models impacts with a coefficient of restitution $0<r<1$
in the sense that $\dot{\beta} \mapsto -r \dot{\beta}$.
Full details of the model including the coefficients of
$\mathbf{A}_{\rm af}$ and $\mathbf{G}$ are given in \Cref{appendix: airfoil model 3dof}.

\begin{figure}[h!]
	\centering
\includegraphics[width=0.5\linewidth]{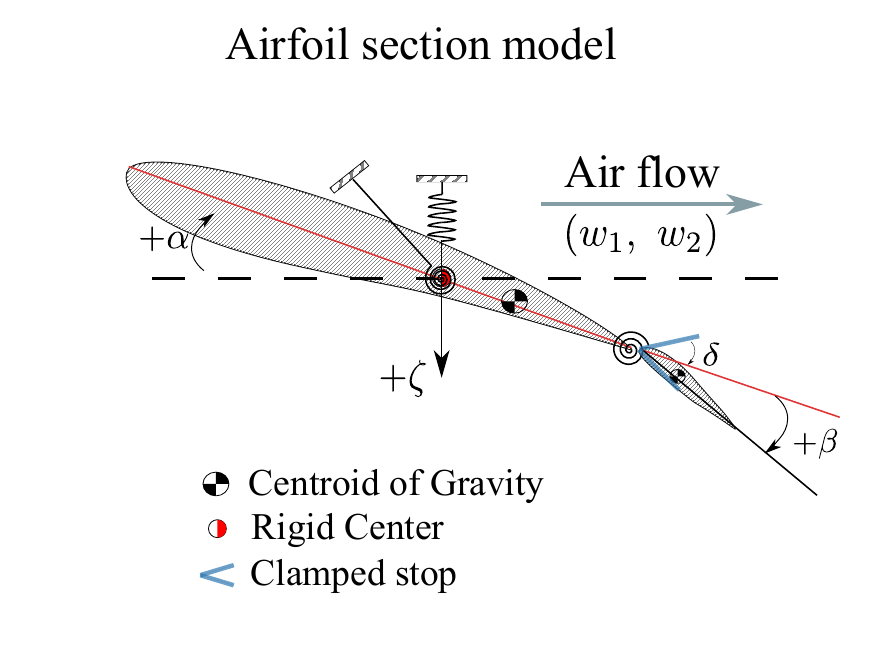}
	\caption{A sketch illustrating the variables of the airfoil model \eqref{eq: airfoil model}.}
	\label{fig: airfoil model}
\end{figure}

\begin{figure}[h!]
	\centering
	\subfloat[]{
		\includegraphics[width=0.45 \linewidth]{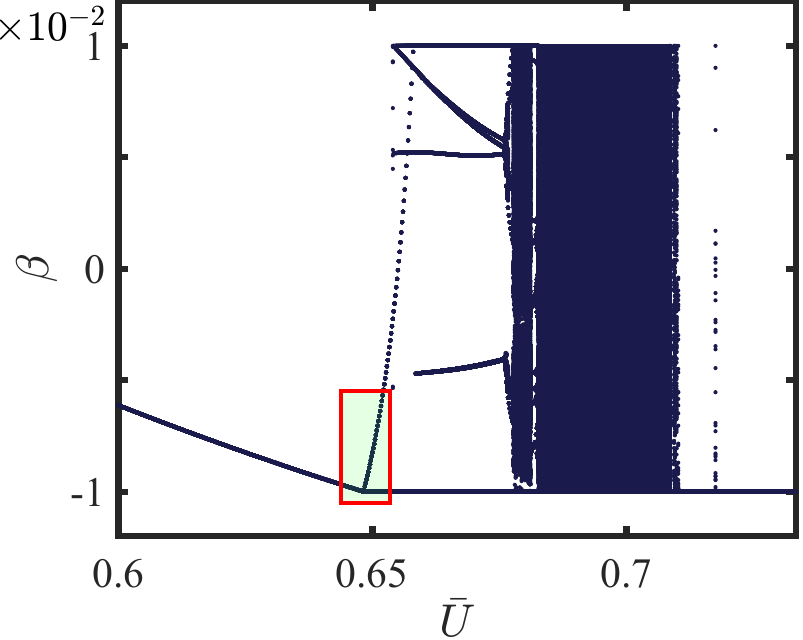}
		\label{fig: full diagram}
	}
	\subfloat[]{
		\includegraphics[width=0.45 \linewidth]{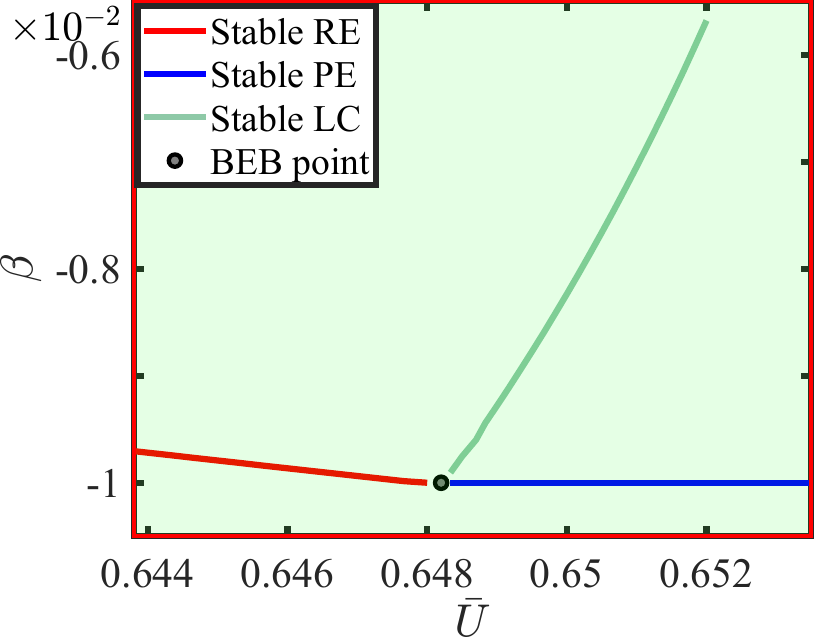}
		\label{fig: zoomin of BEB}
	}
	\caption{Brute force bifurcation diagram of the airfoil model \cref{eq: airfoil model}; (a) the full bifurcation
	diagram capturing various dynamics, (b) zoomed-in part of the first bifurcation from boxed region in (a):
	\emph{PE} -- Pseudo Equilibria, \emph{RE} -- Regular Equilibria. Full equations and parameter definitions are
	given in \cref{appendix: airfoil model 3dof}.}
	\label{fig: BEB diagram}
\end{figure}

\Cref{fig: BEB diagram}, taken from \cite{HoCh23}, shows a typical bifurcation diagram of the airfoil model.
We use the flow velocity $\bar{U}$ as the bifurcation parameter,
while the remaining parameters are fixed at
\begin{equation}
\delta = 0.01, \qquad
r = 0.72, \qquad
\xi_h = \xi_{\alpha} = \xi_{\beta} = 0.02.
\end{equation}
Specifically, we see a stable regular equilibrium collides with the impacting surface $\beta = -\delta$
at $\bar{U} \approx 0.64833$, beyond which the system
has a stable pseudo-equilibrium coexisting with a stable single-impact limit cycle.
The limit cycle subsequently undergoes smooth bifurcations near the impacting surface,
and this is what motivates our studies of codimension-two BEBs.

\subsubsection{BEB-saddle-node bifurcation}\label{subsec:sn-in-airfoil-model}
Using the algorithm and continuation techniques of \cite{HoCh23},
we tracked a curve of saddle-node bifurcations of single-impact limit cycles
in the $(\bar{U},r)$ parameter plane, \Cref{fig:cosim2_BEB_SN}.
We found this curve meets the BEB at the codimension-two point
$\left( \bar{U}_0, r_0 \right) \approx (
0.64833, 0.6292)$.
Here the normal form \eqref{eq:1dSN} has coefficients
\[
a \approx 5.1754, \qquad b \approx 44.1124, 
\]
and in agreement with \eqref{eq:SN_gradient}
the saddle-node bifurcation curve appears to have slope $-\frac{a}{b}$
at the codimension-two point.

\Cref{fig:cosim2_BEB_SNb} shows a typical one-parameter bifurcation diagram
containing both the BEB and saddle-node bifurcations.
At the BEB the stable state changes from a regular equilibrium to a pseudo-equilibrium,
while at the saddle-node bifurcation stable and unstable limit cycles are born
(only the stable limit cycle is indicated in the figure as it was computed
numerically via direct integration of \eqref{eq: airfoil model}).
Note that since the one-parameter bifurcation diagram corresponds to a slice below
the codimension-two point where no limit cycles are created in the BEB,
bifurcation diagram shows a different sequence of dynamics to \Cref{fig:sn_codim2_curve_b}.

\begin{figure}[h!]
\centering
\subfloat
{ \label{fig:cosim2_BEB_SN}
\includegraphics[width=0.45 \linewidth]{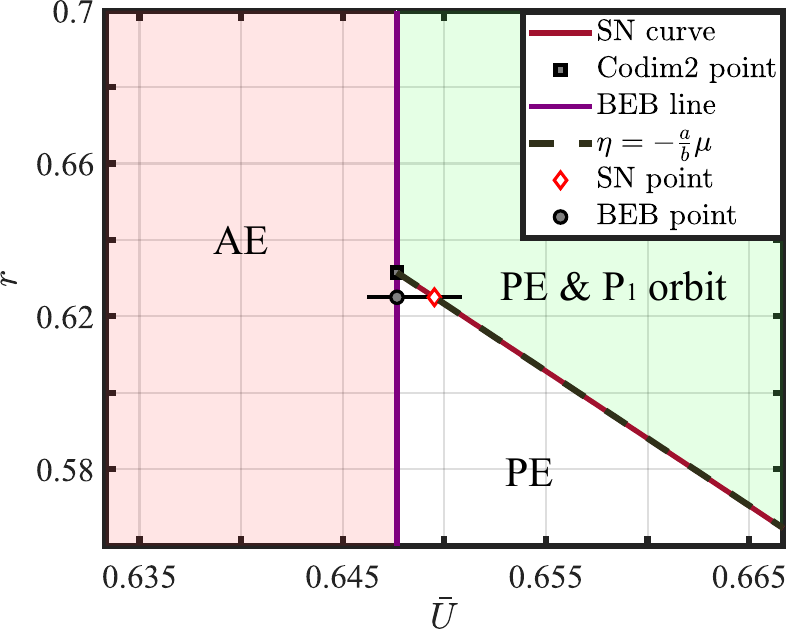}
}
\subfloat
{ \label{fig:cosim2_BEB_SNb}
\includegraphics[width=0.48 \linewidth]{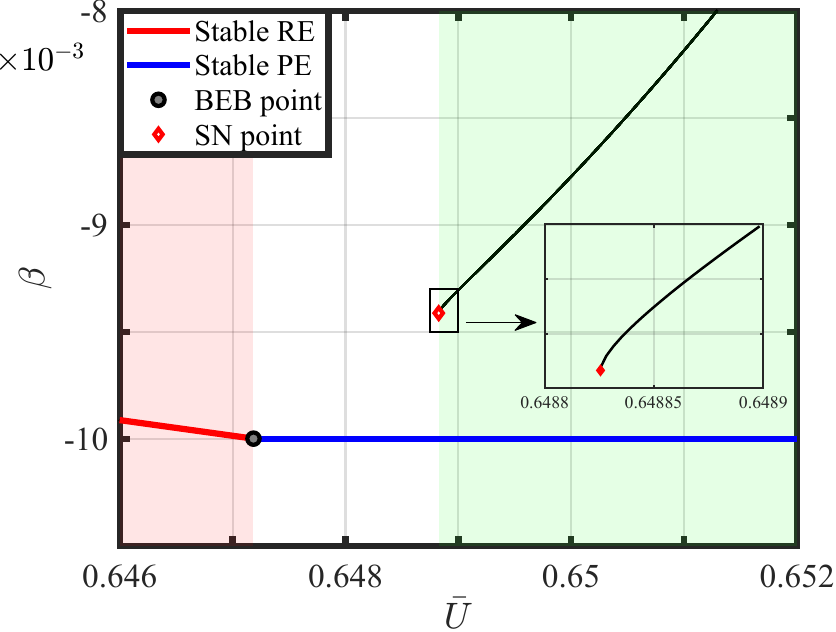}
}
\caption{(a) A two-parameter bifurcation diagram showing the unfolding
of a BEB-saddle-node bifurcation in the airfoil model
\eqref{eq: airfoil model} with $\delta = 0.01$
and $\xi_h = \xi_{\alpha} = \xi_{\beta} = 0.02$.
(b) A one-parameter bifurcation diagram corresponding to setting also $r = 0.625$.
}
\label{fig:unfold_SN_BEB_Airfoil}
\end{figure}

\subsubsection{BEB-period-doubling bifurcation}\label{subsec:pd-in-airfoil-model}
\begin{figure}[h!]
\centering
	\subfloat{
        \label{fig:cosim2_BEB_PD}
		\includegraphics[width=0.48 \linewidth]{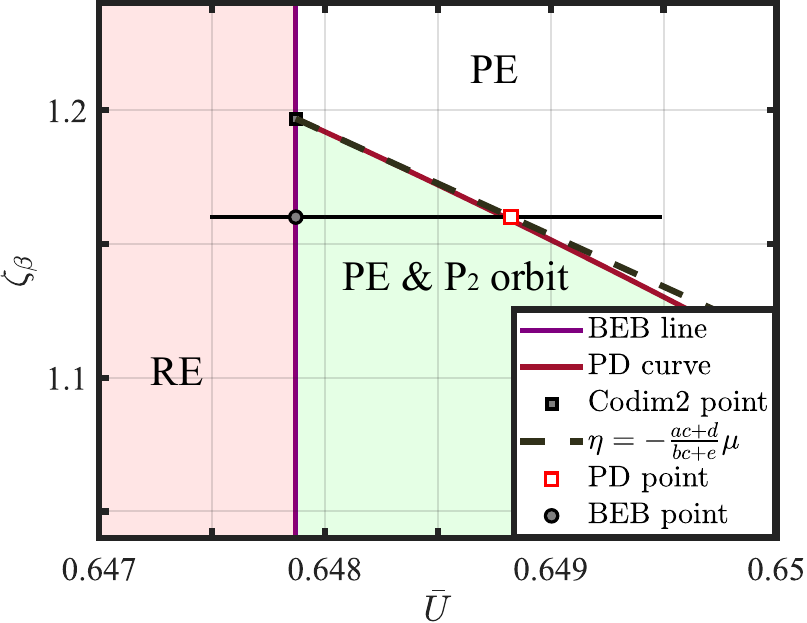}
	}
 \hspace{-1.em}
	\subfloat{
 \label{fig:BEB_PD_diaram}
 \includegraphics[width=0.47 \linewidth]{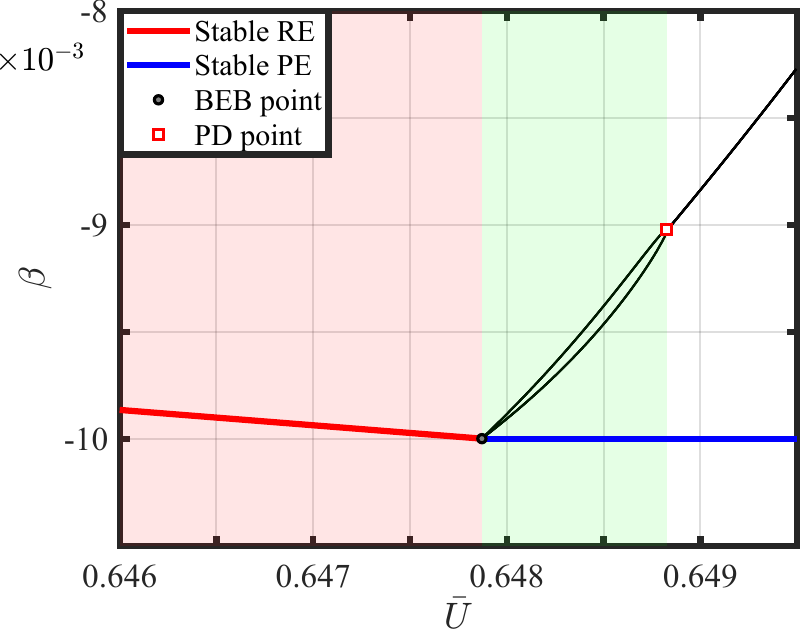}
 }
	\caption{
    (a) A two-parameter bifurcation diagram showing the unfolding
of a BEB-period-doubling bifurcation of the airfoil model
\eqref{eq: airfoil model} with $\delta = 0.01$, $r = 0.72$, and $\xi_h = \xi_{\alpha} = 0.02$.
(b) A one-parameter bifurcation diagram corresponding to setting also $\xi_\beta = 0.0116$.
	}
	\label{fig:unfold_PD_BEB_Airfoil}
\end{figure}

By now fixing $r = 0.72$ and varying $\bar{U}$ and $\xi_\beta$,
we found that a curve of period-doubling bifurcations of single-impact limit cycles
meets the BEB at $\left( \bar{U}_0, \xi_\beta^0 \right) \approx \left( 0.64833, 0.011793 \right)$,
see \Cref{fig:cosim2_BEB_PD}.
Here the normal form \eqref{eq:1dPD} has coefficients
\[
a \approx 16.0757,\quad b \approx -3.5337, \quad c \approx 0.0046, \quad d \approx 0.0595, \quad e \approx 0.2627,
\]
and as expected, the period-doubling curve appears to have
slope $-\frac{a c + d}{b c + e}$ at the codimension-two point.

\Cref{fig:BEB_PD_diaram} shows a nearby one-parameter bifurcation diagram.
The BEB creates a stable two-impact limit cycle
that reverts to a single-impact limit cycle in the period-doubling bifurcation.


\section{Discussion}
\label{sec:conc}

In this paper two codimension-two boundary equilibrium  bifurcations (BEBs) have been analysed for the first time. These are situations where the number and nature of the small-amplitude limit cycles born at BEBs vary due to saddle-node or period-doubling bifurcations of limit cycles.
Analogous situations involving a Neimark-Sacker bifurcation are left for future work.

Future work should also consider global bifurcations
of the small-amplitude limit cycles that emerge from a BEB; see for example the case of a 
homoclinic boundary focus identified in \cite{DeDe11},
and forthcoming work \cite{Shil24} featuring the local birth of Shilnikov-grazing dynamics. 

Clearly there are many more behaviours that can be born locally from a BEB, including, as illustrated in \cref{fig:codim2_PD_chaos}, two different kinds of chaotic attractors. A detailed study of the local creation of chaotic attractors is beyond the scope of the present study, and it would be particularly interesting to
see how these arise beyond a cascade of codimension-two BEB-period-doubling bifurcations.
We expect that the codimension-two BEBs
admit the same types of unfoldings for
Filippov systems and piecewise-smooth continuous ODEs
because in each case the blown-up system is piecewise-linear in the limit $\mu \to 0^+$.

    \section*{Acknowledgments}
Hong Tang was supported by the University of Bristol \& China Scholarship Council joint studentship, ${\rm No
.}202006120007$. David Simpson was supported by Marsden Fund contract MAU2209 managed by Royal Society Te Ap\={a}rangi. The authors thank Mike Jeffrey  for helpful conversations. 
	\printbibliography
    \clearpage
        \appendix
        \section{Numerical computation of the restricted map}
\label{sec:numerics-the-return-map}

\subsection{Parameter-dependent fixed point problem}
\label{sec:para-dependent-map}
Let us start by extending the results in \cite{Yuri98} 
to consider the parameter dependence of the normal form of saddle-node and period-doubling bifurcations more explicitly. Consider a map 
\[
 x \mapsto P(x; \alpha)
\]
where $x \in \mathbb{R}^n$ and $\alpha \in \mathbb{R}$. We will derive formulas for the first few terms of the restriction of $P$
to a centre manifold associated with a saddle-node or period-doubling
bifurcation.


Suppose $P(0;0) = 0$, so that the origin $x=0$ is a fixed point
of $P$ when $\alpha = 0$, and write
\begin{equation}
P(x; \alpha) = \alpha \gamma + A(\alpha) x  + F(x;\alpha),
\end{equation}
where $\gamma = \frac{\partial P}{\partial \alpha}$ is an $n$-dimensional 
constant vector, 
$A(\alpha) \in \mathbb{R}^{n \times n}$, and
$F$ contains all other terms.

At $\alpha =0$, let us suppose that there is a unique multiplier $\lambda_0$ of $A(0)$ such that $|\lambda_0|=1$ ($\lambda_0=1$ for the saddle node case and $\lambda_0=-1$ for period doubling). Let $\lambda(\alpha)$ be a continuous family of such multipliers, and choose 
$w(\alpha)$ and $v(\alpha)$ to be smooth families of the corresponding left and right eigenvectors, respectively, normalized so that $w^\transpose v = 1$ for all small $\alpha$. Henceforth we shall use a subscript $0$ for any variable calculated
at $\alpha=0$ and otherwise drop dependence on $\alpha$ whenever the meaning is clear. 
Thus, $\lambda = w^\transpose A v$ for all small $\alpha$.

Then we make the $\alpha$-dependent projection of $x$ into components in the kernel
and range of $(A-\lambda I)$ via 
$$
x = z v + y, 
\qquad  \mbox{where } z = w^\transpose x, \quad y = x - w^\transpose x,
$$
and write the map $P$ in the new co-ordinates as coordinates 
\begin{align}
\label{eq:map-in-new-coordinate}
	\begin{split}
		z & \mapsto \tilde{z} \coloneq \nu \alpha  + \lambda z +  w^\transpose F(z v + y; \alpha) \\
		y  & \mapsto \tilde{y} \coloneq \kappa \alpha + A y  + F(z v + y;\alpha) - w^\transpose F(z v + y; \alpha) v
	\end{split}
\end{align} 
where $ \nu = w^\transpose \gamma$ is a scalar and  
$\kappa = \gamma - w^\transpose \gamma v \in \mathbb{R}^n$. 

We now follow the projection method of Kuznetsov \cite[pg,~174]{Yuri98} to obtain the normal form. We start by writing the nonlinear part using 
the multivariable Taylor expansion, 
\[
F = \frac{1}{2} B(x,x) + \frac{1}{6} C(x,x,x) + \text{\sc h.o.t.,}
\]
where $B$ is a third-order and $C$ a fourth-order tensor (these are unrelated to the matrices $B$ and $C$ used in the main body of the paper). Then we
we can write  \cref{eq:map-in-new-coordinate} in the form
\begin{align}
\label{eq:expanded_new_map_new_coordinate}
	\begin{split}
		\tilde{z} &= \nu \alpha  + \lambda z +  cz^2 + w^\transpose B(v,y) z  + r z^3 + \text{\sc h.o.t.} \\
		\tilde{y} &= \kappa \alpha + A y  + \frac{1}{2} H_{02} z^2 +  B(v ,  y) z - z w^\transpose  B(v,y) v 
	 + \frac{1}{2}  B(y,y) - \frac{1}{2}  w^\transpose B(y,y) v + \text{\sc h.o.t.},
	\end{split}
\end{align} 
where 
$$
c = \frac{1}{2} w^ \transpose B(v,v), \quad r = \frac{1}{6}w^ \transpose C(v,v,v), \quad  H_{02} = B(v,v) - w^\transpose B(v,v) v.
$$

Next, let us suppose we write the center manifold of  \cref{eq:expanded_new_map_new_coordinate} as
\begin{equation} 
	y = V(z, \alpha) \coloneq w_{01} \alpha + w_{11} \alpha z +\frac{1}{2} w_{20} z^2 + \frac{1}{2}w_{02} \alpha^2 + \text{\sc h.o.t.},
\label{eq:CM}
\end{equation}
where the $w_{ij}$ are vectors that we need to calculate. We then 
substitute \cref{eq:CM} into \cref{eq:expanded_new_map_new_coordinate}, in order to write  the restricted map as 
\begin{equation} \label{eq:restricted_map_1D}
		\tilde{z}  = \nu\alpha + \lambda(\alpha) z +c z^2 +   w^\transpose B(v ,  w_{01})  \alpha z  + (\frac{1}{2}w^\transpose B(v ,  w_{20})  + r) z^3  + \text{\sc h.o.t.}.
\end{equation}
Similarly, for the $y$ component we obtain
 \begin{align} \label{eq:restricted-y-map}
 	\begin{split}
 		\tilde{y} &= \kappa \alpha  + A w_{01} \alpha + A w_{11} \alpha z + \frac{1}{2} A w_{20} z^2 +\frac{1}{2} H_{02} z^2  \\ 
 		&+ ( B(v ,  w_{01}) - w^\transpose B(v ,  w_{01}) v) \alpha z +   \frac{1}{2}   B(y,y)   - \frac{1}{2}  w^\transpose  B(y,y) v +  \text{\sc h.o.t}
 	\end{split}
 \end{align}

In order to obtain the unknown coefficient vectors $w_{01}, w_{11}, w_{20}$, we have the to solve the homological equation 
\begin{equation}
\tilde{y} = V(\tilde{z}, \tilde \alpha)
\label{eq:homolog}
\end{equation}
After, substituting the expression \cref{eq:restricted_map_1D} for $\tilde{z}$ into $V$ defined by \cref{eq:CM}, 
we find the first few terms of the right-hand side of \cref{eq:homolog} are
\begin{equation*}
V(\tilde{z}, \alpha)  = w_{01} \alpha  
 + w_{11} \alpha (\nu \alpha + \lambda(\alpha) z) 
+\frac{1}{2} w_{20} (\nu \alpha + \lambda(\alpha) z + \cdots)^2  
+ \text{\sc h.o.t.}, 
\end{equation*}
We now solve \cref{eq:homolog} term by term, by matching coefficients of 
$\alpha$,  $\alpha z$ and $z^2$. From this we obtain 
\begin{align}
	\kappa + A w_{01} &= w_{01} \\
	A w_{11} + B(v ,  w_{01}) - w^\transpose B(v ,  w_{01}) v & =  w_{11} \lambda_0  + w_{20} \nu \lambda_0\\
	A w_{20} + H_{02} &  = w_{20}.
\end{align}
Looking at the first equation, note that in the case of the saddle-node, 
the definition of $\kappa$ ensures that it is in the range of 
$(A-I)$, therefore we can find a solution up to a scalar constant. 
The same is true for all the $w_{ij}$, for both the saddle-node and period-doubling cases. In particular, we find 
\begin{description}
    \item[SN case:]  $\lambda_0 = 1$.
        \begin{align*} 
        w_{01} & = (I - A)^{{\rm INV}}(\gamma - w^\transpose \gamma v), \\
        w_{11} & =  (I - A)^{{\rm INV}} \left[ B(v ,  w_{01}) - w^\transpose B(v ,  w_{01}) v  - w_{20} \nu \right], \\
        w_{20} & = (I - A)^{{\rm INV}}H_{02};
        \end{align*}
    \item[PD case:] $\lambda_0 = -1$.
    \begin{align*}
    w_{01} & = (I - A)^{-1}(\gamma - w^\transpose \gamma v)  = (I - A)^{-1} \gamma  - \frac{w^\transpose \gamma}{2} w, \\
    w_{11} & =  -( I + A)^{{\rm INV}} \left[ B(v ,  w_{01}) - w^\transpose B(v ,  w_{01}) v + w_{20} \nu  \right], \\
    w_{20} & = ( I - A)^{-1} H_{02}.
    \end{align*}
\end{description}
Here ${\rm INV}$ represents a pseudo-inverse computation, which is needed when the matrix is not invertible, but the right-hand side is in the range of the operator. In that case, the solution is only defined up to a scalar (see \cite[pg,~184]{Yuri98} for details). 

It is helpful in what follows to notice that 
$$ 
\frac{\partial^2 \tilde{z} }{ \partial \alpha \partial z} = w^\transpose{\rm D} \frac{\partial P}{\partial \alpha}v + w^\transpose B(v,w_{01}).
$$

\subsection{Evaluation of coefficients of restricted maps}

We now show how the above results enable us to find closed form expressions for all the coefficients in \eqref{eq:1dSN} and \eqref{eq:1dPD} in terms of
the (full) Poincar\'e map $P(u;\mu,\eta)$ and its derivatives. In particular, following the logic of the previous section, with $\alpha$ replaced by 
either $\mu$ or $\eta$, we find 
%
%
the coefficients of \eqref{eq:1dSN} and \eqref{eq:1dPD}
are given by
\[
a = w^\transpose \frac{\partial P}{\partial \mu}, \quad
b = w^\transpose \frac{\partial P}{\partial \eta}, \quad 
c= \frac{1}{2}w^\transpose \big[ (\rD^2 P)(v,v) \big];
\]
with
\[
\quad 
d = w^\transpose D \frac{\partial P}{\partial \mu} v + w^\transpose B\big(v ,  (I - A)^{{\rm INV}}(\frac{\partial P}{\partial \mu} - a v) \big)  , \quad
e = w^\transpose D \frac{\partial P}{\partial \eta} v + w^\transpose B\big(v ,  (I - A)^{{\rm INV}}(\frac{\partial P}{\partial \eta} - b v) \big) ,
\]
\[m = \frac{1}{2}w^\transpose \frac{\partial^2 P}{\partial \mu^2}, \quad
n = \frac{1}{2} w^\transpose \frac{\partial^2 P}{\partial \eta^2}, \quad
f =  \frac{1}{6} w^\transpose C(v,v,v)  + \frac{1}{2}  w^\transpose B(w, B(w, (I -A )^{\rm INV} H_{02}) ;
\]
in the saddle-node case, and
\[
\quad 
d = w^\transpose D \frac{\partial P}{\partial \mu} v + w^\transpose B\big(v ,  (I - A)^{-1}(\frac{\partial P}{\partial \mu} - a v) \big)  , \quad
e = w^\transpose D \frac{\partial P}{\partial \eta} v + w^\transpose B\big(v ,  (I - A)^{-1}(\frac{\partial P}{\partial \eta} - b v) \big) ,
\]
\[
f =  \frac{1}{6} w^\transpose C(v,v,v) - c^2 + \frac{1}{2}  w^\transpose B(w, B(w, (I -A )^{-1} B(v,v)) .
\]
in the period-doubling case.

\section{Approximation of the bifurcating limit cycles}
\label{sec:amp-approximation}

Here, we provide approximate expressions for the amplitudes of the bifurcating limit cycles.
We use the incoming (impact) velocities
of the limit cycles to measure their amplitude.
These are obtained by evaluating the velocity
relative to the impacting surface at the corresponding
fixed points $\hat{u}$ of the Poincar\'e map.

Given the expansion of the return map for the scaled system \eqref{eq:hybridSystemScaled}, the locus of fixed points can be parameterized as a curve $ u = h(s; \mu, \eta)$, as in \eqref{eq:1dSN} and \eqref{eq:1dPD}.  Let $\hat{s} \in \mathbb{R}$ be such that $\hat{u} = h(\hat{s};0,0)$, and $z = s - \hat{s}$ evaluated from $z = q(z; \mu, \eta)$. Importantly, the right eigenvector $v$ is tangent to the curve $h$ at $\hat{u} = h(\hat{s};0,0)$.

In blown-up coordinates, the limit cycle amplitude is
$\hat{\mathcal{A}} = \tilde{v}^- = \nabla H(y)^\transpose F(y)$,
and
\begin{equation}
\label{eq:amp_definition}
\hat{\mathcal{A}} = C^{\transpose}Ay + C^{\transpose}M +  \mu C^{\transpose}A_1 y + \mu C^{\transpose}M_1  +o(\mu),
\end{equation}
according to \cref{eq:hybridSystemScaled}.
The amplitude in the coordinates of the original system
is then simply given by $\mathcal{A} = \hat{\mathcal{A}} \mu$.

For small $\mu > 0$ the fixed point is near $\hat{u}$,
so we write $u_p(\mu) = \hat{u}+\delta u$. We have
$y = \zeta^{-1}(u_p) = \begin{bmatrix}
    0\\ u_p
\end{bmatrix}$, so
\begin{equation}
    \delta u = z v + \text{\sc h.o.t.}
\end{equation}
Then from \eqref{eq:amp_definition} we have
\begin{equation}
\label{eq:BmathcalA}
    \hat{\mathcal{A}} = \ell_0 + \mu \ell_1 + L_{01} z  + L_{11} \mu z +\text{\sc h.o.t.}
\end{equation}
where 
\begin{align*}
     & \ell_0 = C^\transpose (A \zeta^{-1}(\hat{u}) + M), \quad \ell_1 = C^\transpose (A_1 \zeta^{-1}(\hat{u}) + M_1), \quad L_{01} = C^\transpose A \zeta^{-1}(v),  \quad
  L_{11} = C^\transpose A_1 \zeta^{-1}(v).
\end{align*}


To prove \cref{th:hybridSN,th:hybridPD}
for the codimension-two saddle-node and period-doubling scenarios,
we used Taylor expansions centred at $(\mu,\eta) = (0,0)$.
Now consider a nearby point $(\mu,\eta) = (\mu_c,\eta_c)$
that belongs to the bifurcation curve of saddle-node
or period-doubling bifurcations,
and compute the limit cycle amplitude
as a function of $\mu$ with $\eta = \eta_c$ fixed.

\subsection{The saddle-node case} 
    \label{sec:amp-approx-SN}

    To approximate the amplitude we truncate
    \eqref{eq:1dSN} to
    \begin{equation}\label{eq:BEB-SN-truncated-map}
         q(z;\mu,\eta) = z + a \mu + b \eta + c z^2 + d \mu z + e \eta z + m \mu^2 + n \eta^2.
    \end{equation}
    The fixed point equation for \eqref{eq:BEB-SN-truncated-map}
    is quadratic; by the quadratic formula the fixed points are
\begin{align}
z  = \mathcal{Y}_{1,2} -\frac{d \mu + e \eta }{ 2c },
\label{eq:SN_amp_approx_Scaled}
\end{align}
where
  \[
   \mathcal{Y}_{1,2} = \pm \sqrt{ \frac{(d\mu + e \eta)^2}{4c^2}   - \frac{a\mu  + b\eta}{c} 
   - \frac{m \mu^2 + n \eta^2}{c}}.
  \]
Then by \eqref{eq:BmathcalA}
if the blown-up system has two limit cycles
for $0 < \mu < \mu_c$, these have amplitudes
\begin{equation}
    \hat{\mathcal{A}} = \ell_0 + \mu \ell_1 +  L_{01} \left( \pm \sqrt{-\frac{a}{c} (\mu - \mu_c) + \frac{(d^2 - 4mc)(\mu - \mu_c)^2}{4c^2} } -\frac{d (\mu - \mu_c)}{2c} \right),
\end{equation}
where $\ell_0 + \ell_1 \mu_c$ is the amplitude of the fixed point $\hat{u}$ at $(\mu_c, \eta_c)$,
and $L_{01}$ captures the
variation in $z$.
Multiplying this by $\mu$ gives
\begin{equation}
\label{eq:SN_amp_LCO_approx_OP}
\mathcal{A} = \ell_0 \mu + \ell_1 \mu^2 + \mu  L_{01} \left( \pm \sqrt{-\frac{a}{c} (\mu - \mu_c) + \frac{(d^2 - 4mc)(\mu - \mu_c)^2}{4c^2} } -\frac{d (\mu - \mu_c)}{2c} \right).
\end{equation}

    \subsection{The period-doubling case} 
    
With a coordinate shift
		$$
		\tilde{z} = z + \frac{a \mu + b \eta}{d\mu + e \eta -2},
		$$
the map \eqref{eq:1dPD} can be written in the new coordinate as
		$$
		\check{p}(\tilde{z};\mu, \eta) = -\tilde{z} + \check{b} (\mu, \eta) \tilde{z} + \check{c} (\mu, \eta) \tilde{z}^2 + f \tilde{z}^3 + \text{\sc {h.o.t}},
		$$
		where $\displaystyle \check{c} = c - \frac{3f(a\mu + b\eta)}{d\mu + e\eta - 2}$, $\displaystyle \check{b} =  d \mu + e \eta - \frac{2c(a \mu + b \eta)}{d \mu + e \eta -2}$.
	We can transform the map into a normal form with a smooth coordinate
		transformation
		$$
		\tilde{z} = y + \frac{\check{c}}{\lambda^2 - \lambda}y^2
		$$
		where $\displaystyle \lambda = -1 + \check{b}$.
		We can now write
		\begin{equation}
             \label{eq:PD_normal_form}
		\check{p}(y) =  \lambda y + \hat{c} y^3 + \cO(y^4),
		\end{equation}
		where $\displaystyle  ~\hat{c} = \frac{2\check{c}^2}{\lambda^2 - \lambda}+ f $.
		Further, after we apply the rescaling, 
		$$
		y = \frac{\gamma}{\sqrt{|\hat{c}|}},
		$$
		the system takes the normal form with new coordinate
		$$
		\tilde{\eta} = -(1 - \check{b}) \gamma + \chi \gamma^3 + \cO(\gamma^4),
		$$
		where $\chi= \operatorname{sign}(\hat{c})$. Near the period-doubling bifurcation, the period-two solution
        has points
		$\gamma_{1,2} = \pm \sqrt{-\check{b}} + o \big( (|\mu| + |\eta|)^{1/2} \big) $ \cite[pg,125]{Yuri98}.
		In $z$-coordinates the fixed point of the map is
		$$
		\tilde{z}=h_1(\mu,\eta) = \frac{a \mu + b \eta}{2} + o(|\mu| + |\eta|)
		$$
while
 \begin{equation}
     \tilde z_{1,2} = \pm \sqrt{- \frac{\check{b}}{\hat{c}}} + \frac{\check{c}\, \check{b}}{|\hat{c}|(\lambda^2 -\lambda)} +h_1(\mu, \eta)+ o(|\mu| + |\eta|) 
\end{equation}
are the points of the period-two solution.

Again, we treat $\eta = \eta_c$
as fixed and assume the period-two solution exists
for $0 < \mu < \mu_c$.
Then the fixed point corresponds to a limit cycle with amplitude
\begin{equation}
    \label{eq:P1_amp_LCO_approx_blowup}
\hat{\mathcal{A}} = \ell_0  +  \ell_1 \mu  + L_{01}\frac{a (\mu - \mu_c)}{2}  + o(|\mu| + |\mu -\mu_c|),
\end{equation}
while the period-two solution corresponds to a limit cycle with amplitudes
\begin{equation}
\label{eq:PD_amp_approx_scaled}
\hat{\mathcal{A}} = \ell_0 + \ell_1 \mu + L_{01}(\pm \sqrt{- \frac{(ca+d)(\mu - \mu_c)}{c^2+f}} + \frac{c(ca+d)}{ 2(c^2 + f)}(\mu - \mu_c) + \frac{a(\mu - \mu_c)}{2} ) +  o(|\mu| + |\mu -\mu_{c}|).
\end{equation}
Multiplying these by $\mu$ gives the amplitudes
in the original system coordinates.

        \section{Full equations of motion for the airfoil model}
\label{appendix: airfoil model 3dof}

The model studied in \cite{HoCh23,PeterAIAA} is a reduced-order model of two-dimensional airfoil within a constant air stream. A full derivation can be found in \cite{PeterAIAA}; here we simply specify the equations in full. The three mechanical degrees of freedom are
$\alpha$, $\beta$, and $\zeta$. The first two of these represent the angular displacement (pitch) of the airfoil and flap respectively, while $\zeta = \frac{h}{b}$
 is the dimensionless displacement in the heave degree of freedom, normalized by the semi-chord $b$. The parameter
$\bar{U} = \frac{U}{\omega_{\alpha} b}$ is a dimensionless measure of the magnitude of the free stream air velocity
approaching the airfoil, and $\delta$
characterizes the amount of flap freeplay.

Using Lagrangian mechanics one obtains the equations of motion of the mechanical degrees of freedom in the form
\begin{equation}
    \overline{\boldmath M}
    \begin{bmatrix}
        \ddot \zeta  \\
        \ddot \alpha \\
        \ddot \beta
    \end{bmatrix}
    +
    \overline{\boldmath C}
    \begin{bmatrix}
        \dot \zeta  \\
        \dot \alpha \\
        \dot \beta
    \end{bmatrix}
    +
    \overline{\boldmath K}
    \begin{bmatrix}
        \zeta  \\
        \alpha \\
        \beta
    \end{bmatrix}
    =
    \begin{bmatrix}
        L/ (mb)           \\
        T_\alpha / m{b^2} \\
        T_\beta /  (mb^2)
    \end{bmatrix}
    +
    \begin{bmatrix}
        \bf F
    \end{bmatrix}
    \label{eq:airfoil_model}
\end{equation}
$$
\mbox{where} \quad \overline {\boldmath M} =
\begin{bmatrix}
    1 & \bar x_\alpha & \bar x_\beta
    \\
    \bar x_\alpha & \bar r_\alpha ^2 & \bar r_\beta ^2 +
    \bar x_\beta (\bar c - \bar a) \\
    \bar x_\beta & \bar r_\beta ^2 + \bar x_\beta (\bar c - \bar a) & {\bar r_\beta ^2}
\end{bmatrix},
$$
$$
\overline{\boldmath K}  =
\begin{bmatrix}
    \omega _h^2  	& 0        				 	& 0               \\
    0           & {\omega _\alpha ^2\bar r_\alpha ^2} 		& 0             \\
    0           & 0                       	& {\omega _\beta ^2\bar r_\beta ^2}
\end{bmatrix}
\mbox{ and }
\overline{\boldmath C}=(\Phi^T)^{-1}
\begin{bmatrix}
    2\xi_h \omega _h 		 &	 0                       & 0                                       \\
    0                & 2\xi_\alpha \omega _\alpha \bar r_\alpha ^2 	 & 0                                       \\
    0                & 0                         & 2\xi_\beta\omega _\beta \bar r_\beta ^2
\end{bmatrix}
\Phi^{-1},
$$
where  $\Phi$ is an eigenvector matrix defined by $(\overline{\boldmath K}-\omega^2
\overline{\boldmath M})\phi_i=0$,
$\Phi=[\phi_1 \cdots \phi_n]$,
and $\Phi^T \overline{\boldmath M} \Phi=\boldmath I$. Also,  $L$, $T_\alpha$, and $T_\beta$ define state-dependent generalised
aerodynamic forces, defined below, and $\bf F$ represents other external
generalised forces (set to zero in the current model, except for preload $1\% \cdot \delta k_{\beta}$  in the
component corresponding to the rotational flag degree).
Each $\xi_i$, for
$i\in\{h,\alpha,\beta\}$,
corresponds to the mode-proportional structural damping ratios for each degree of freedom; by default we set
$\xi_i=\xi=0.02$ for each degree of freedom \cite{wright2008}.

The unsteady aerodynamics $L, T_{\alpha}, T_{\beta} $ are given as
	\begin{subequations}
		\allowdisplaybreaks
		\begin{align}
			\begin{split}
				L           &= \pi {\rho _a}{b^2}
				\left(
				{\ddot h + V\dot \alpha  - b\bar a\ddot \alpha  - \frac{V}{\pi }{T_4}\dot \beta  - \frac{b}{\pi }{T_1}\ddot \beta }
				\right)
			\\
				&+ 2\pi {\rho _a}Vb
				\left(
				Q_a(\hat \tau ){\phi _w}(0) - \int_0^{\hat \tau } {Q_a} (\sigma )\frac{\rm{d} {\phi _w}(\hat \tau  - \sigma )}{{\rm d}\sigma }{\rm d}\sigma
				\right),
			\label{eq:timedomain aerodynamics eq1}
		\end{split}
			\\ \vspace{4ex}
		\begin{split}
				T_\alpha    & = \pi {\rho _a}{b^2}
				\left[
				b\bar a\ddot{h} - Vb\left( \frac{1}{2} - \bar a \right)\dot \alpha  - b^2\left( \frac{1}{8} + {\bar a^2} \right)\ddot \alpha  - \frac{V^2}{\pi }\left( T_4 + T_{10} \right)\beta  \right.
			\\
				& + \left. \frac{Vb}{\pi}\left(- T_1 + T_8 + (\bar c - \bar a)T_4 - \frac{1}{2}T_{11} \right)
				\dot \beta  + \frac{b^2}{\pi }\left( {T_7 + (\bar c - \bar a)T_1} \right)\ddot \beta
				\right]
			\\
				& + 2\pi {\rho _a}V{b^2}\left( \bar a + \frac{1}{2} \right)\left( Q_a(\hat \tau ){\phi _w}(0) - \int_0^{\hat \tau } {Q_a}(\sigma )\frac{\rm{d}{\phi _w}(\hat \tau  - \sigma )}{{\rm d}\sigma }{\rm d}\sigma  \right),
				\label{eq:timedomain aerodynamics eq2}
			\end{split}
			\\ \vspace{4ex}
			\begin{split}
				T_\beta    &= \pi {\rho _a}{b^2}
				\left[
				\frac{b}{\pi }{T_1}\ddot h + \frac{Vb}{\pi }\left( 2T_9 + T_1 - \left( \bar a - \frac{1}{2} \right)T_4 \right)\dot \alpha
				- \frac{2b^2}{\pi}T_{13}\ddot \alpha  \right.
			\\
				& \left.  - \left( \frac{V}{\pi } \right)^2  \left(T_5 - T_4 T_{10} \right)\beta  + \frac{Vb}{2 \pi ^2}T_4T_{11}\dot \beta  + \left( \frac{b}{\pi } \right)^2 T_3 {\ddot \beta}  \right]
			\\
				& - {\rho _a}V{b^2} T_{12} \left( {Q_a}(\hat \tau ){\phi _w}(0) - \int_0^{\hat \tau } {Q_a} (\sigma ) \frac{{\rm d}{\phi _w}(\hat \tau  - \sigma )}{{\rm d}\sigma }{\rm d}\sigma  \right).
				\label{eq:timedomain aerodynamics eq3}
			\end{split}
		\end{align}
		\label{eq:timedomain aerodynamics}
	\end{subequations}
In order to approximate the unsteady aerodynamics,
we use the exponential
approximation to the Theodorsen functions
$$
\phi(\tau)=1-a_{1} {\rm e}^{-b_{1} \tau}-a_2 {\rm e}^{-b_2 \tau},
$$
as introduced by Jones \cite{Jones1938}; see \cite{wright2008} for a
derivation and definitions of the coefficients $a_{1,2}$ and $b_{1,2}$.  We then introduce the augmented variables
\begin{equation}
{ w_1}(t )
    = \int_0^{t }
    Q_a{\rm{e}^{- {b_1}(t  - \sigma )}}{\rm d}\sigma, \quad
        { w_2}(t ) = \int_0^{t }  Q_a {\rm{e}^{- {b_2}(t  - \sigma )}}{\rm d}\sigma,
    \label{eq:augmented variables}
\end{equation}
to calculate the aerodynamic forces $L$, $T_{\alpha}$, and $T_{\beta}$ in terms of
feedback from the structural motion, where
 $$
	{Q_a} =\left(V\alpha  + \dot h + b\left( \frac{1}2 - \overline{a} \right)\dot \alpha  + \frac{V}{\pi }{T_{10}}\beta
+ \frac{b}{{2\pi }}{T_{11}}\dot \beta     \right ).
$$
Then with $X_s=[\zeta,\alpha,\beta]^{\top}$, for the structural variables, and $w_p=[w_1 , w_2]^{\top}$ for the
augmented parametric variables,
the full coupled system is
\begin{equation}
    \begin{aligned}
        \dot X_s &= \dot X_s
        \\
        M{\ddot X_s}&=- K X_s- C\dot{X}_s -D_{w} w_p
        \\
        {\dot w_p}    &=E_q X_s+E_{qd} \dot{X}_s+E_{w} w_p
        \\
    \end{aligned}
    \label{state 1}
\end{equation}
where
$$
M = \overline{M} - \eta M_{nc}, \;    K=\overline{K} - \eta (U/b)^2 (K_{nc} + 0.5R_c S_{c1}), \;  C = \overline{C} -
\eta (U/b)({B}_{nc} + 0.5R_c S_{c2}),
$$
$$
D_{\omega} = \eta (U/b)R_c \begin{bmatrix}
                               a_1 b_1 (U/b)^2  & a_2 b_2 (U/b)
\end{bmatrix}, \;
E_q  = (U/b)
\begin{bmatrix}
    S_{c1}; {S}_{c1}
\end{bmatrix},
\;
E_{qd} =
\begin{bmatrix}
{S}
    _{c2};{S}_{c2}
\end{bmatrix},
$$
$$
E_{\omega}=\begin{bmatrix}
               -b_1 & 0   \\
               0   & -b_2
\end{bmatrix},\quad
\eta  = 1/\pi \mu, \quad \mbox{and} \quad
\mu  = m/\pi {\rho _a}{b^2},
$$
\begin{equation*}
    \begin{aligned}
        \boldmath{M}_{nc} & = \begin{bmatrix}
                                  - \pi      & \pi \bar a                         & T_1       \\
                                  \pi \bar a & - \pi \left(1/8 + \bar a^2 \right) & - 2T_{13} \\
                                  T_1        & - 2T_{13}                          & T_3/\pi
        \end{bmatrix},
        & \boldmath{B}_{nc}& =\begin{bmatrix}
                                  0 & - \pi              & T_4          \\
                                  0 & \pi (\bar a - 0.5) & - T_{16}     \\
                                  0 & - T_{17}           & - T_{19}/\pi
        \end{bmatrix}
    \end{aligned}
\end{equation*}
\begin{equation*}
    \begin{aligned}
        \boldmath{R}_c & =
        \begin{bmatrix}
            - 2\pi              \\
            2\pi (\bar a + 0.5) \\
            - T_{12}
        \end{bmatrix}, \quad
        \boldmath{K}_{nc} & = \begin{bmatrix}
                                  0 & 0 & 0                  \\
                                  0 & 0 & { - {T_{15}}}      \\
                                  0 & 0 & { - {T_{18}}/\pi }
        \end{bmatrix}, \\
        \boldmath{S}_{c1} & = \begin{bmatrix}
                                  0  & 1  & \dfrac{T_{10}}{\pi }
        \end{bmatrix},
        \quad \boldmath{S}_{c2} & = \begin{bmatrix}
                                        1& {0.5 - \bar a} &
                                        \dfrac{T_{11}}{2\pi }
        \end{bmatrix},
    \end{aligned}
\end{equation*}
with all $T_i$ constants given in \cite{THEODORSEN1935}.

 Finally we transform the differential-integral equations \cref{eq:airfoil_model} into the following system of
first-order ODEs:
\begin{equation}
    \begin{bmatrix}
        \dot { X }_s \\
        \ddot { X}_s \\
        \dot { w}_{\rm p}
    \end{bmatrix}
    =
    \begin{bmatrix}
        \boldmath{0}_{3 \times 3} &
        \boldmath{I}_{3 \times 3} &
        \boldmath{0}_{3 \times 2} \\
        - {\boldmath M}^{-1} {\boldmath K} &
        - {\boldmath M}^{-1} {\boldmath C} &
        - {\boldmath M}^{-1} {\boldmath D}\\
        {\boldmath E}_q &
            {\boldmath E} _{qd} &
            {\boldmath E}_w
    \end{bmatrix}
    \begin{bmatrix}
    {X_s }
        \\
        {\dot {X }_s}\\
        { w}_{\rm p}
    \end{bmatrix}
    +
    \begin{bmatrix}
    {\boldmath{0}_{3 \times 1}}
        \\
        {-{ {\boldmath M} ^{-1}} {\boldmath F} ( {X_s})}\\
        {\boldmath{0}_{2 \times 1}}
    \end{bmatrix}.
    \label{eq22}
\end{equation}

The physical parameters used in this study are given in \cref{tab: parameters of airfoil model}.
\begin{table}[ht]
    \label{tab: parameters of airfoil model}
    \centering
    \caption{Parameter Definition}
    $
    \begin{array} {|c|c|c|c|c|c|}
        \hline
        \multicolumn{6}{|c|}{\mbox{Physical parameters}} \\
        \hline
        b & {\omega}_h & {\omega}_{\alpha} & {\omega}_{\beta} & \rho_{\rm a}
        &  m \\
        \hline
        0.3 ~\rm m & 50~ {\rm rad/s} & 100 ~{\rm rad/s} & 0~ {\rm rad/s} & 1.225~ \rm{ kg/ \mathrm{mm^3}}
        &  1.5 ~{\rm kg}\\
        \hline
        a_1 & a_2 & b_1 & b_2 & \xi_i, i=h,\alpha,\beta
        & r\\
        \hline
        0.165  & 0.0455 & 0.335 & 0.3 & 2\%
        & 0.72
        \\
        \hline
        \multicolumn{6}{|c|}{\mbox{Dimensionless parameters}} \\
        \hline
        \overline{a} & \overline{c} & \overline{x}_{\alpha} & \overline{x}_{\beta} & \overline{r}_{\alpha}^{2} &
        \overline{r}_{\beta}^{2} \\ [2pt]
        \hline
        -0.4 & 0.6 & 0.2 & 0.0125 & 0.25
        & 0.00625 \\
        \hline
    \end{array}
    $
\end{table}

For convenience, we also specify here the numerically
evaluated matrices needed to compute the normal form
\cref{eq:IHS_series}.
At the BEB $\bar{U} =0.64833$ and
$\delta = 0.01~{\rm rad}$, we obtained
$$ \mathbf{A} = [\mathbf{A}_1~ \mathbf{A}_2]$$
where
$$
\mathbf{A}_1  = \left(\begin{array}{ccc}
                          0 & 0 & 0 \\ 0 & 0 & 0 \\ 0 & 0 & 0 \\ -2.9340e+03 & 2.3800e+03 & -31.8848 \\ 2.5143e+03 &
                          -1.4569e+04 & -126.9591 \\ -1.5787e+03 & 3.9373e+04 & 119.8092 \\ 0 & 0 & 0 \\ 0 & 64.8330 &
                          35.6462
\end{array}\right),
$$
$$\mathbf{A}_2 = \left(\begin{array}{cccccccc}
                           1 & 0 & 0 & 0 & 0 \\  0 & 1 & 0 & 0 & 0 \\  0 & 0 & 1 & 0 & 0 \\ -4.1409 & -1.7578 & -0.2147
                           & -118.8655 & -29.0256 \\  3.3583 & -8.2454 & -1.0773 & 157.7863 & 38.5297 \\  -3.2826 & 17
                           .0083 & -1.9570 & -328.2203 & -80.1478 \\ 0 & 0 & 0 & 0 & 1 \\1 & 0.9000 & 0.1487 & -57.3753
                           & -22.3998
\end{array}\right), $$
and, regarding the reset map,
$$
\mathbf{C} = \left(\begin{array}{c}
                              0\\ 0\\ 1\\ 0\\ 0\\ 0\\ 0\\ 0
\end{array}\right),~ \mathbf{B} = (1+r)\left(\begin{array}{c}
                                                 0\\ 0\\ 0\\ 0.0030\\ -0.0774\\ 1\\ 0\\ 0
\end{array}\right).
$$



\end{document}